\documentclass[12pt]{amsart}
\usepackage[english]{babel}
\usepackage{graphicx}
\usepackage{color}
\usepackage{amsmath,amssymb,hyperref,srcltx}
\newtheorem{pro}{Proposition}[section]
\newtheorem{teo}{Theorem}[section]

\newtheorem{cor}{Corollary}[section]

\newtheorem{lem}{Lemma}[section]

\theoremstyle{definition}
\newtheorem{defi}{Definition}[section]

\theoremstyle{remark}
\newtheorem{rem}{Remark}[section]

\begin{document}

\title[]
{Fixed points of nilpotent actions on ${\mathbb S}^{2}$}

\author{Javier Rib\'{o}n}
\address{Instituto de Matem\'{a}tica, UFF, Rua M\'{a}rio Santos Braga S/N
Valonguinho, Niter\'{o}i, Rio de Janeiro, Brasil 24020-140}
\thanks{e-mail address: javier@mat.uff.br}
\thanks{MSC-class. Primary: 37C85, 37E30; Secondary: 37C25, 55M20, 57S25}

\thanks{Keywords: local diffeomorphism, derived length, solvable group, nilpotent group}
\maketitle

\bibliographystyle{plain}
\section*{Abstract}
We prove that a nilpotent subgroup of
orientation preserving $C^{1}$ diffeomorphisms of
${\mathbb S}^{2}$ has a finite orbit of cardinality at most two.
We also prove that a finitely generated nilpotent subgroup of orientation
preserving $C^{1}$ diffeomorphisms of ${\mathbb R}^{2}$
preserving a compact set has a global fixed point. These
results generalize theorems of Franks, Handel and Parwani for the
abelian case.

We show that a nilpotent subgroup of
orientation preserving $C^{1}$ diffeomorphisms of ${\mathbb S}^{2}$
that has a finite orbit of odd cardinality also has a global fixed point.
Moreover we study the
properties of the two-points orbits of
nilpotent fixed-point-free subgroups of
orientation preserving $C^{1}$ diffeomorphisms of ${\mathbb S}^{2}$.
\section{Introduction}
Consider a closed surface $S$ with non-vanishing Euler characteristic.
The classical Poincar\'{e} theorem asserts that any $C^{1}$ vector field defined
in $S$ has a singular point.
Lima \cite{Lima-com} proved that this result can be extended to any finite family
of pairwise commuting $C^{1}$ vector fields and then to any continuous action
of a connected abelian finite-dimensional Lie group $G$ on $S$. Plante generalized
the previous result for connected nilpotent finite-dimensional Lie groups \cite{Plante-fix}.

 Subgroups of $\mathrm{Diff}_{0}^{1}(S)$ are
the natural generalization of continuous actions of connected Lie groups
(we denote by $\mathrm{Diff}^{1}(S)$ and $\mathrm{Diff}_{0}^{1}(S)$
the group of $C^{1}$ diffeomorphisms of $S$ and its subgroup of isotopic
to the identity elements respectively).
The Lefschetz fixed point formula implies that any element of
$\mathrm{Diff}_{0}^{1}(S)$ has a fixed point.
An intermediate case between the study of continuous actions of Lie algebras
and the general case is obtained by considering subgroups $G$ of
$\mathrm{Diff}^{1}(S)$
whose generators are close to the identity map in
the $C^{1}$ topology. The orbits of $\phi \in G$
are contained in poligonal curves that are almost tangent to the trajectories of
a vector field $X$ associated to $\phi$ (cf. \cite{Bonatti-com}).
Bonatti took profit of
this flow-like behavior to prove that any commutative subgroup $G$
of $\mathrm{Diff}_{0}^{1}(S)$
(where $S$ is a compact surface with non-vanishing Euler characteristic $\chi (S)$)
whose generators
are $C^{1}$-close to the identity has a global fixed point \cite{Bonatti-s2} \cite{Bonatti-com},
i.e. a common fixed point for all the elements of $G$.
This result was generalized by Druck, Fang and Firmo for nilpotent groups
of diffeomorphisms of the sphere
in the $C^{1}$-close to the identity setting \cite{D-F-F}.

 Let us say a word about the $C^{1}$ hypothesis.
A recent result of Parkhe shows that a finitely generated virtually nilpotent subgroup
of orientation-preserving homeomorphisms of ${\mathbb S}^{1}$ is always topologically
conjugated to a group of $C^{1}$ diffeomorphisms \cite{Parkhe-smooth}. Nevertheless
continuous and $C^{1}$ dynamics are quite different in dimension $2$.
For instance the homeomorphism $f: {\mathbb S}^{2} \to {\mathbb S}^{2}$
defined by $f(r,\theta)=(r, \theta + 1/r)$ in polar coordinates, is not topologically
conjugated to a diffeomorphism since points rotate arbitrarily fast around the origin
when $r \to 0$.

 Another property of dimension $1$ is that the elements of finitely generated
nilpotent subgroups $\Gamma$
of orientation-preserving diffeomorphisms of ${\mathbb S}^{1}$ can be supposed arbitrarily
close of isometries.
Moreover precisely, Navas showed that there exists a path $(\phi_{t})_{t \in [0,\infty)}$
of homeomorphisms of ${\mathbb S}^{1}$ such that
$\phi_{t} \circ f \circ \phi_{t}^{-1}$ tends to a rotation
in the $C^{1}$ topology when $t \to \infty$ for any $f \in \Gamma$ \cite{Navas-rappro}.
An analogue for the $2$-sphere would allow to use the ideas in  \cite{Bonatti-com} and
\cite{D-F-F} to study nilpotent actions by $C^{1}$ diffeomorphisms.
However such an analogue is false.
Consider irrational numbers $\alpha \neq \beta$ and $C^{\infty}$ functions
$g:(0,\infty) \to (0,\infty)$ and $h:(0,\infty) \to {\mathbb R}$ such that
$g(r)=r$ for any $r \in (0,1] \cup [2, \infty)$, $g(r) > r$ for any $r \in (1,2)$,
$h(r) = \alpha$ for any $r \in (0,1]$ and $h(r) = \beta$ for any $r \in [2,\infty)$.
The $C^{\infty}$ diffeomorphism $f: {\mathbb S}^{2} \to {\mathbb S}^{2}$ defined by
$f(r,\theta)=(g(r), \theta + h(r))$ is a pseudo-rotation since
its unique periodic points are the fixed points $r=0$ and $r=\infty$.
Anyway there is no sequence of topological conjugates of $f$
that converge to a rotation $R$ in the $C^{1}$ topology since
then the rotation number of $R$ would be equal to both $\alpha$ and $\beta$.

 A natural idea to extend the global fixed point results to
general subgroups of
$\mathrm{Diff}_{0}^{1}({\mathbb S}^{2})$ is replacing flows (or almost flows) with
isotopies.
In this way Franks, Handel and Parwani solve completely the problem of
existence of fixed points for commutative groups. They prove that an abelian
subgroup of $\mathrm{Diff}_{0}^{1}(S)$ ($\chi (S) < 0$) has always global fixed points
(and even more that there are fixed points in the identity lift of the group)
\cite{FHP-g}. In the case $\chi (S)>0$, i.e. $S={\mathbb S}^{2}$, the group
$\mathrm{Diff}_{0}^{1}({\mathbb S}^{2})$ coincides with the group
$\mathrm{Diff}_{+}^{1}({\mathbb S}^{2})$ of orientation-preserving $C^{1}$ diffeomorphisms.
In contrast with the negative Euler characteristic case,
there exist fixed-point-free commutative
groups of M\"{o}bius transformations (the group generated by two involutions with
disjoint fixed point sets such that any involution permutes the fixed points of
the other involution). Franks, Handel and Parwani characterize the subgroups
of $\mathrm{Diff}_{+}^{1}({\mathbb S}^{2})$ that have global fixed points.
\begin{teo}
\label{teo:FHP}
\cite{FHPs}
For any abelian subgroup $G$ of $\mathrm{Diff}_{+}^{1}({\mathbb S}^{2})$
there is a subgroup $H$ of index
at most two such that $\mathrm{Fix}(H) \neq \emptyset$.
Moreover, $\mathrm{Fix}(G) \neq \emptyset$ if and only if $w(\phi,\eta) = 0$
for all $\phi, \eta \in G$.
\end{teo}
The invariant $w(\phi,\eta)$ is defined for commuting
orientation-preserving homeomorphisms of ${\mathbb S}^{2}$
and belongs to ${\mathbb Z}/2 {\mathbb Z}$
(cf.  subsection \ref{subsec:cond}).

We generalize Theorem \ref{teo:FHP} for nilpotent groups.
\begin{teo}
\label{teo:sphere}
Let $G$ be a nilpotent subgroup of $\mathrm{Diff}_{+}^{1}({\mathbb S}^{2})$.
Then $G$ has  a finite orbit of cardinality at most $2$.
\end{teo}
Let us remark that the existence of finite orbits for
nilpotent groups of real analytic diffeomorphisms of the sphere was already known
(Ghys \cite{Ghys-identite}).

The proof of Theorem \ref{teo:FHP} relies on the study of lifts of the group
to universal coverings of the complementary of some closed invariant sets and
a fixed point lemma whose generalization for the nilpotent case follows.
\begin{teo}
\label{teo:plane}
Let $G \subset \mathrm{Diff}_{+}^{1}({\mathbb R}^{2})$ be a finitely generated
nilpotent subgroup
that preserves a non-empty compact set. Then $\mathrm{Fix}(G)$ is not empty.
\end{teo}
A version of Theorem \ref{teo:plane} for groups of orientation-preserving
homeomorphisms of ${\mathbb R}^{2}$
is proved in \cite{FRV:arxiv}.
The $C^{1}$ condition is replaced with a condition of global type
on a set of generators of $G$. The generators are required to be close to the identity map in some sense.
Under such hypotheses existence and localization of fixed points are obtained.

The next theorem is a corollary of Theorem \ref{teo:plane}.
It is useful since some constructions in the paper
naturally satisfy  the hypotheses in Theorem \ref{teo:plane2}.
\begin{teo}
\label{teo:plane2}
Let $G \subset \mathrm{Diff}_{+}^{1}({\mathbb R}^{2})$ be a
nilpotent subgroup. Suppose that there exists $\phi \in G$
such that $\mathrm{Fix}(\phi)$ is a non-empty compact set.
Then $\mathrm{Fix}(G)$ is not empty.
\end{teo}
The following results are applications of Theorems \ref{teo:plane2} and
\ref{teo:plane}.
\begin{cor}
\label{cor:d2}
If $G$ is a nilpotent subgroup of $\mathrm{Diff}_{+}^{1}({\mathbb D}^{2})$
then $\mathrm{Fix}(G)$ is non-empty.
\end{cor}
\begin{teo}
\label{teo:carlit}
(Generalization of Cartwright-Littlewood theorem).
Let $G \subset \mathrm{Diff}_{+}^{1}({\mathbb R}^{2})$ be a nilpotent subgroup
preserving a full continuum $K$ of the plane.
Then $G$ has a global fixed point in $K$.
\end{teo}
Let $G$ be an abelian subgroup of $\mathrm{Diff}_{+}^{1}({\mathbb S}^{2})$.
The invariant $w$ (cf. Theorem \ref{teo:FHP}) defines a symmetric map
$w:G \times G \to {\mathbb Z}/2 {\mathbb Z}$, that is
a morphism of groups in every component and such that
$w_{| H \times H} \equiv 0$ is equivalent to the existence of a global fixed point of $H$
for any subgroup $H$ of $G$ \cite{FHPs}.
In general it is not possible to define an invariant $w$
satisfying the previous properties if $G$ is a nilpotent subgroup of
$\mathrm{Diff}_{+}^{1}({\mathbb S}^{2})$.
Consider $G= \langle 1/z, i z \rangle \subset \mathrm{PGL}(2,{\mathbb C})$ as a counterexample;
it is isomorphic to the dihedral group $D_{4}$ and in particular $G$ is a
non-abelian nilpotent group.
We have $w(1/z,i z) = w(1/z,-z) =1$ since $1/z$ does not have common fixed points
with neither $iz$ nor $-z$. On the other hand
$w(1/z,-z) = w(1/z, (iz) \circ (iz)) = 2 w (1/z,iz)=0$ since $w$ is a morphism of groups
in the second component, providing a contradiction.
Anyway the following parity theorem is an example of a result that
admits a very simple proof in the abelian setting by using the properties of $w$
and it still can be generalized to the nilpotent case.
\begin{teo}
\label{teo:odd}
Let $G \subset \mathrm{Diff}_{+}^{1}({\mathbb S}^{2})$ be a nilpotent subgroup.
Suppose that there exists a finite $G$-invariant set $F$ such that
$\sharp F$ is odd. Then $\mathrm{Fix}(G)$ is not empty.
\end{teo}
The previous theorem can be also interpreted as a result on fixed-point-free
groups. Any finite orbit of a fixed-point-free nilpotent subgroup of
$\mathrm{Diff}_{+}^{1}({\mathbb S}^{2})$ has an even number of elements.
We study the properties of these groups in section \ref{section:fpf}.
We say that a finite orbit of $G$ with $p$ elements is a $p$-orbit of $G$.
Consider $2$-orbits ${\mathcal O}_{1}$ and ${\mathcal O}_{2}$ of a
subgroup $G$ of $\mathrm{Homeo}_{+}({\mathbb S}^{2})$,
where $\mathrm{Homeo}_{+}({\mathbb S}^{2})$ is the group of orientation-preserving
homeomorphisms of ${\mathbb S}^{2}$.
We say that ${\mathcal O}_{1}$ and ${\mathcal O}_{2}$ belong to the same class if we have
\[ \phi_{|{\mathcal O}_{1}} \equiv Id \Leftrightarrow \phi_{|{\mathcal O}_{2}} \equiv Id \]
for any $\phi \in G$.
Curiously fixed-point-free nilpotent subgroups of $\mathrm{Diff}_{+}^{1}({\mathbb S}^{2})$
also have few $2$-orbits.
\begin{teo}
\label{teo:scar}
(Scarcity of $2$-orbits).
Let $G$ be a fixed-point-free
nilpotent subgroup of $\mathrm{Diff}_{+}^{1}({\mathbb S}^{2})$.
Then the number of classes of $2$-orbits is either
$1$ or $3$. Moreover, if $G$ is commutative it is equal to $3$.
\end{teo}
\section{Technical setup}
Let us introduce some of the tools that we are going to use
throughout the paper.
\subsection{Groups}
The main results are proved
by induction on the class of nilpotency $k$
of $G$.
\begin{defi}
Let $G$ be a group.
We denote by $Z^{(l)}(G)$ the $l^{\mathrm{th}}$ member of the upper central
series of $G$. More precisely we have $Z^{(0)}(G) = \{Id\}$ and
\[ Z^{(k+1)}(G) = \{ \phi \in G : \phi \eta \phi^{-1} \eta^{-1} \in Z^{(k)}(G) \ \forall \eta \in G \}  \]
for $k \geq 0$.
If $G$ is nilpotent then
$c=\min \{ k \in {\mathbb N} \cup \{0\} : Z^{(k)}(G) = G \}$
is the nilpotency class of $G$.
\end{defi}
\begin{defi}
Let $G$ be a nilpotent group.
We say that $G$ is in ${\mathcal G}_{k,p}$ if $G$ is a nilpotent group of
nilpotency class $\leq k+1$ such that $G = \langle Z^{(k)}(G), \alpha_{1}, \hdots, \alpha_{p} \rangle$
for some $\alpha_{1}, \hdots, \alpha_{p} \in G$.
\end{defi}
\begin{rem}
Given a group $G= \langle Z^{(k)}(G), \alpha_{1}, \hdots, \alpha_{p} \rangle$
in ${\mathcal G}_{k,p}$ the group
$\langle Z^{(k)}(G), \alpha_{1}, \hdots, \alpha_{p-1} \rangle$ belongs to ${\mathcal G}_{k,p-1}$.
\end{rem}
\subsection{Nielsen classes}
A frequent trick in this paper is to consider a lift $\tilde{G}$ of
$G \subset \mathrm{Diff}_{+}^{1}({\mathbb S}^{2})$, deducing properties of
$\tilde{G}$ and interpreting them in terms of $G$.
\begin{defi}
Let $G$ be a subgroup of the group $\mathrm{Homeo}_{+}(S)$ of orientation-preserving
homeomorphisms of a manifold $S$.
We say that a subgroup $\tilde{G}$ of $\mathrm{Homeo}_{+}(\tilde{U})$
is a lift of $G$ if
\begin{itemize}
\item $\tilde{U}$ is the universal covering of a $G$-invariant open set $U$ of $S$.
We denote by $\pi: \tilde{U} \to U$ the covering transformation.
We define $G_{|U}$ as
the group  induced by $G$ by restriction to $U$.
\item There exists a unique $\tilde{\phi} \in \tilde{G}$
such that $\phi \circ \pi = \pi \circ \tilde{\phi}$ for any $\phi \in G_{|U}$.
Reciprocally given $\psi \in \tilde{G}$ there exists
$\phi \in G$ such that $\phi \circ \pi = \pi \circ \psi$.
\end{itemize}
\end{defi}
\begin{rem}
Notice that $\tilde{G}$ is not in general the group of all lifts of elements of $G$.
For example the unique lift of $\{ Id_{U} \}$ is $\{ Id_{\tilde{U}} \}$.
The second property implies that the map $\tau: G_{|U} \to \tilde{G}$
defined by $\tau(\phi)= \tilde{\phi}$ is an isomorphism of groups.
It is clear that $\mathrm{Fix}(\tilde{G}) \neq \emptyset$
implies $\mathrm{Fix}(G) \cap U \neq \emptyset$.
\end{rem}
We use Nielsen classes
to analyze the properties of the fixed point sets of the elements of
a lift.
\begin{defi}
Let $\phi \in \mathrm{Homeo}_{+}(S)$. Consider a compact $\phi$-invariant
set $C$. Let $x,y \in \mathrm{Fix}(\phi) \setminus C$. We say that
$x$ is {\it Nielsen equivalent to} $y$ {\it relative to} $C$ if there exists
an arc $\rho$ from $x$ to $y$ such that $\rho$ and $\phi (\rho)$ are homotopic
relative to endpoints in $S \setminus C$.
\end{defi}
\begin{defi}
Let $\phi \in \mathrm{Homeo}_{+}(S)$, a compact $\phi$-invariant
set $C$ and $x \in \mathrm{Fix}(\phi) \setminus C$. Consider the connected
component $N$ of $S \setminus C$ containing $x$ and its universal covering
$\tilde{N}$. We say that $\tilde{\phi}:\tilde{N} \to \tilde{N}$ is a lift
of $\phi$ for the Nielsen class of $x$ rel $C$ if $\mathrm{Fix}(\tilde{\phi})$
contains a lift of $x$.
\end{defi}
Notice that $x$ and $y$ are Nielsen equivalent relative to $C$
if they belong to the same connected component $N$ of
$S \setminus C$ and there exists
some lift $\tilde{\phi}:\tilde{N} \to \tilde{N}$ of
$\phi$ for both the Nielsen classes of $x$ and $y$ rel $C$.
Indeed in such a case the lifts of $\phi$ for the Nielsen class of $x$ rel $C$
coincide with the lifts of $\phi$ for the Nielsen class of $y$ rel $C$.

The next definitions are used in Proposition \ref{pro:lif1} (which is
indeed the Proposition 4.2 of \cite{FHPs}). They are required to
relate properties of the images of a subgroup $G \subset \mathrm{Diff}_{+}^{1}({\mathbb S}^{2})$
in the group of classes of
homeomorphisms modulo isotopy (relative to some compact $G$-invariant set)
with properties of $G$.
\begin{defi}
\label{def:det}
Let $\eta_{t}$ be an isotopy  rel $C$ from $\eta_{0}=\phi$ to $\eta_{1}$.
Consider $x \in \mathrm{Fix}(\phi) \setminus C$ and a lift $\tilde{\phi}$
of $\phi$ for the Nielsen class of
$x$ rel $C$. There exists a unique continuous lift $\tilde{\eta}_{t}$ of $\eta_{t}$
such that $\tilde{\eta}_{0}=\tilde{\phi}$.
{\it The Nielsen class relative to $C$ determined by the Nielsen class of} $x$ {\it and} $\eta_{t}$ is by
definition the Nielsen class determined by $\tilde{\eta}_{t}$, i.e.
the elements of the Nielsen class are the projections of elements of
$\mathrm{Fix}(\tilde{\eta}_{t})$.
\end{defi}
\begin{defi}
We say that the Nielsen class of $x \in \mathrm{Fix}(\phi) \setminus C$ rel $C$
is {\it compactly covered} if there exists some lift $\tilde{\phi}:\tilde{N} \to \tilde{N}$
(and hence every lift)
of $\phi$ for the Nielsen class of $x$ such that
$\mathrm{Fix}(\tilde{\phi})$ is compact.
\end{defi}
\begin{defi}
Suppose that $P$ is an isolated puncture of $N$.
We say that the Nielsen class of $x \in \mathrm{Fix}(\phi) \setminus C$ rel $C$
{\it peripherally contains} $P$ if
there is a properly immersed
ray $\rho \subset N$ based at $x$ that converges to $P$ such that
$\phi (\rho)$ is properly homotopic to $\rho$ relative
to $x$ in $N$.
\end{defi}
\begin{defi}
Let $\phi \in \mathrm{Homeo}_{+}({\mathbb R}^{2})$. Suppose that $C \subset {\mathbb R}^{2}$
is a non-empty compact $\phi$-invariant set with $C \cap \mathrm{Fix}( \phi ) = \emptyset$.
We define $P(C,\phi)$ as the union of the Nielsen classes of $\mathrm{Fix}(\phi)$ rel $C$
that do not peripherally contain $\infty$.
\end{defi}
Periodic points of orientation-preserving homeomorphisms of ${\mathbb R}^{2}$
turn around fixed points \cite{Gamba} \cite{Calvez:tourne}. The next
result is a generalization of such phenomenon for compact invariant sets
with no fixed points.
\begin{pro}
\label{pro:niel}
\cite{FHPs}[Proposition 5.3]
Let $\phi \in \mathrm{Homeo}_{+}({\mathbb R}^{2})$. Suppose that $C \subset {\mathbb R}^{2}$
is a non-empty compact $\phi$-invariant set with $C \cap \mathrm{Fix}( \phi ) = \emptyset$.
Then $P(C, \phi)$ is a non-empty compact set.
\end{pro}
\subsection{Realization of subgroups of the mapping class group}
An idea in fixed point theory is that irreducible models of elements or
groups contain the minimum possible number of fixed points for elements in
the same isotopy class. Hence it is natural to study the image $[G]$ of a group $G$
in certain mapping class groups. Anyway if we replace blindly any element of
the group $G$ with a simpler one in its isotopy class we lost the algebraic structure;
the new group is in general not nilpotent even if $G$ was.
We need to realize $[G]$ as a nilpotent group of orientation-preserving homeomorphisms.
\begin{defi}
\label{def:irr}
Consider a compact surface $S$ of negative Euler characteristic and with
finitely many punctures. We say that a subgroup $A$ of the mapping class
group of $S$ is {\it irreducible} if every element of $A$ is either
pseudo-Anosov or of finite order.
\end{defi}
\begin{lem}
\label{lem:model}
Let $S$ be a compact surface $S$ of negative Euler characteristic and
with finitely many punctures. Let $A$ be an irreducible nilpotent subgroup
of the mapping class group of $S$. Then there exists a homomorphism
$\sigma: A \to \mathrm{Homeo}_{+}(S)$ such that
$\sigma(\theta)$ is pseudo-Anosov or periodic and
is isotopic to $\theta$ for any
$\theta \in A$. Moreover, if $A$ contains a pseudo-Anosov class then
$\sigma(A)$ is generated by the finite group of torsion elements of $\sigma(A)$
and a pseudo-Anosov homeomorphism.
In particular $A$ and $\sigma(A)$ are virtually cyclic.
\end{lem}
\begin{proof}
The group $A$ is solvable. Hence it is virtually abelian and finitely generated
(Theorems A and B of \cite{BLM}). If there is no pseudo-Anosov class in the
center of $A$ then $A$ is finite (every finitely generated nilpotent group with
finite center is finite, cf. \cite{Baum}[Lemma 0.1, page 3]).
In such a case the group $A$ can be realized as a subgroup of isometries of
a Riemannian metric defined in $S$, this is a consequence of the solution of the
Nielsen's realization problem by Kerchoff \cite{Kerck}.
The realization of the centralizer of a pseudo-Anosov mapping class is solved in
Lemma 2.10 of \cite{FHPs}.
Indeed the centralizer of a pseudo-Anosov class $[f_{0}]$ is generated by its subgroup of
finite order elements and a pseudo-Anosov class $[f_{1}]$.
Therefore the group generated by a pseudo-Anosov class $[f_{0}] \in Z^{(1)}(A)$
is a finite index normal subgroup of $A$.
\end{proof}
\begin{rem}
\label{rem:pt2}
In this paper we always apply Lemma \ref{lem:model} to
finitely punctured spheres.
In such a case we can describe
the nature of the
subgroup of finite order elements of $\sigma (A)$.
The list of finite groups of orientation-preserving homeomorphisms
defined in the sphere is well known.
It coincides up to conjugacy with the list of finite groups of
orientation-preserving isometries.
The list is composed by cyclic groups $C_{n}$, dihedral groups $D_{n}$
with $2n$ elements,
$A_{4}$, $S_{4}$ and $S_{5}$ (cf. \cite[p. 40]{Shurman}).
The unique nilpotent groups in that list are the cyclic groups and
the dihedral groups $D_{2^{n}}$ for $n \in {\mathbb N}$.
Examples of realizations of the groups $D_{2^{n}}$
are presented in section \ref{subsec:cond}.

Notice that if $\sigma (A)$ has non-trivial finite order elements so
does the center of $\sigma (A)$. Since finite order elements have exactly two
fixed points, we deduce that $\sigma (A)$ has either a finite orbit of
$2$ elements or $2$ fixed points. Moreover
if $\sigma (A)$ has a global fixed point
then it has a second one. If $\sigma (A)$ has no torsion elements then
it is generated by a pseudo-Anosov element
and it has at least $2$ fixed points.
Indeed the Lefschetz index of a fixed point of pseudo-Anosov homeomorphism
is at most $1$.
%
%
\end{rem}
\subsection{Isotopies and fixed points}
The next results are key ingredients to translate results for simple models
(provided by Lemma \ref{lem:model}) to the initial nilpotent subgroup of
$\mathrm{Diff}_{+}^{1}({\mathbb S}^{2})$.
\begin{pro}
\label{pro:lif1}
\cite{FHPs}[Proposition 4.2]
Let $\theta \in \mathrm{Homeo}_{+}(S)$ where $S$ is a compact, maybe finitely punctured, surface.
Consider a compact $\theta$-invariant set $C$. Let $M_{1}$ be a $\theta$-invariant essential
connected subsurface of $S \setminus C$
of finite type and negative Euler characteristic.
Suppose that $\theta_{|M_{1}}$ is either pseudo-Anosov relative to the puncture set or periodic non-trivial.
Let $f_{t}:S \to S$ be an isotopy rel $C$ from $f_{0}=\theta$ to $f_{1}=\phi$.
Suppose that there exists a point $z \in \mathrm{Fix}(\theta)$ in the interior of $M_{1}$.
Consider the $\theta$-Nielsen class $\mu$ of $z$ rel $C$.
Let $\nu$ be the $\phi$-Nielsen class rel $C$ determined by $\mu$ and $f_{t}$.
Then $\nu$ is a non-trivial, compactly covered Nielsen class that does not peripherally contain
any punctures.
\end{pro}
We apply the previous proposition to identify elements of nilpotent groups that have a lift
whose fixed point set is a non-empty compact set. In the group setting let us consider
a finitely generated
subgroup ${\mathcal F}$ of $\mathrm{Homeo}_{+}(S)$ for a finitely punctured compact surface.
Let $K$ be a compact ${\mathcal F}$-invariant set. As in the previous proposition we consider
a connected essential ${\mathcal F}$-invariant (modulo isotopy rel $K$) subsurface
$M_{1} \subset S \setminus K$ of finite
type and negative Euler characteristic. There exists a well-defined morphism of groups
$\eta: {\mathcal F} \to \mathrm{MCG}(M_{1})$ defined by $\eta(\phi) = \theta_{|M_{1}}$
where $\theta \in \mathrm{Homeo}_{+}(S)$ satisfies $\theta(M_{1})=M_{1}$ and is isotopic
to $\phi$ rel $K$. We denote by $M$ the connected component of $S \setminus K$ that
contains $M_{1}$.
\begin{pro}
\label{pro:lif2}
Suppose that $\eta({\mathcal F})$ is an irreducible non-trivial nilpotent group.
Let $\Theta$ be the
model provided for $\eta({\mathcal F})$ by Lemma \ref{lem:model}.
Suppose that $\Theta$ has a global fixed point in the interior of $M_{1}$.
Consider an element $\phi$ of ${\mathcal F}$ such that $\eta(\phi) \neq Id$.
Then there exists a lift $\tilde{\mathcal F}$ defined in the universal covering of $M$
such that $\mathrm{Fix}(\tilde{\phi})$ is a non-empty compact set.

Suppose further that
${\mathcal F} = \langle Z^{(k)}({\mathcal F}), \phi_{1},\hdots,\phi_{p} \rangle$
is a nilpotent group of nilpotency class less than or equal to $k+1$.
Then there exist $\alpha_{1}$, $\hdots$, $\alpha_{p}$ in
${\mathcal F}$ such that
$\tilde{\mathcal F} = \langle Z^{(k)}(\tilde{\mathcal F}), \tilde{\alpha}_{1},\hdots,\tilde{\alpha}_{p} \rangle$
and $\mathrm{Fix}(\tilde{\alpha}_{j})$ is a non-empty compact set for
any $1 \leq j \leq p$.
\end{pro}
\begin{proof}
Let $\tilde{M}_{1}$, $\tilde{M}$ be the universal coverings of $M_{1}$, $M$
respectively.
Fix a lift $\tilde{z}$ of $z$ in $\tilde{M}_{1}$.
There exists a unique lift $\tilde{\theta}$ of $\theta$ to
$\tilde{M}_{1}$ such that $\tilde{\theta}(\tilde{z}) = \tilde{z}$
for any $\theta \in \Theta$. We obtain a lift
$\tilde{\Theta}$ of $\Theta$.

Given $\phi \in {\mathcal F}$ consider an element
$\theta \in \mathrm{Homeo}_{+}(S)$ that is isotopic to
$\phi$ rel $K$ and satisfies $\theta_{|M_{1}} \in \Theta$.
The restriction $\theta_{|M_{1}}$ is uniquely defined.
By lifting the isotopy from $\theta$ to $\phi$ we obtain
a lift $\tilde{\phi}_{1}: \tilde{M}_{1} \to \tilde{M}$.
Since $M_{1}$ has negative Euler characteristic, $\tilde{\phi}_{1}$ is uniquely defined.
The map $\tilde{\phi}_{1}$
extends uniquely to a homeomorphism $\tilde{\phi}: \tilde{M} \to \tilde{M}$.
Moreover
$\phi \mapsto \tilde{\phi}$ defines a lift $\tilde{\mathcal F}$ of
${\mathcal F}$.
Proposition \ref{pro:lif1} implies that
$\mathrm{Fix}(\tilde{\phi})$ is a non-empty compact set if
$\eta(\phi)$ is not a trivial class.

There exists $\beta \in {\mathcal F}$ such that $\eta(\beta)$ is non-trivial.
Moreover we can suppose that $\beta$ belongs to
$Z^{(k)}({\mathcal F}) \cup \{\phi_{1},\hdots,\phi_{p}\}$.
We define $\alpha_{j}=\phi_{j}$ if $\eta (\phi_{j})$ is non-trivial and
$\alpha_{j}= \phi_{j} \beta$ otherwise.
\end{proof}
\section{Fixed point theorem for the plane}
Theorem \ref{teo:plane} for commutative groups in
$\mathrm{Diff}_{+}^{1}({\mathbb R}^{2})$ was proved in \cite{FHPs}.
In this section we extend the result to nilpotent groups.
More precisely,
we prove the result for subgroups of $\mathrm{Diff}_{+}^{1}({\mathbb R}^{2})$
in ${\mathcal G}_{k,p}$ by a double induction process, first on $p$ then on $k$.

Let us explain the idea of the proof.
Let $\mathrm{Homeo}_{+}({\mathbb S}^{2})$ be the group of
orientation-preserving homeomorphisms of ${\mathbb S}^{2}$.
We consider a nilpotent subgroup $G$ of $\mathrm{Diff}_{+}^{1}({\mathbb R}^{2})$
as a subgroup of $\mathrm{Homeo}_{+}({\mathbb S}^{2})$
where all the elements fix
the point at infinity. At this point we consider a $G$-invariant closed set
$K$, for instance the fixed point set of a normal subgroup $H$ of $G$.
The idea is taking profit of a decomposition {\it \`{a} la} Thurston
(with improvements along the lines in \cite{FHPs})
of the classes of
isotopy rel $K$. Roughly speaking, the set of isolated points of $K$ is finite and suitable
to apply a Thurston decomposition whereas the set of accumulation points $K'$
of $K$ satisfies that any element of $H$ is isotopic to the identity map in
the neighborhood of $K'$ rel $K'$ (this property is a consequence of the smoothness of the
elements of $G$ and does not hold true for homeomorphisms). Hence it is natural to require
$K$ to be compact in ${\mathbb R}^{2}$ since the elements of $G$ are not $C^{1}$
at $\infty$. A technical difficulty arises since in general the set
$\mathrm{Fix}(H)$ is not compact. The following statements
$A_{k,p}$, $B_{k,p}$, $C_{k,p}$ have hypotheses requiring the group
to have varying degrees of compactness:

${\bf A_{k,p}}$: Theorem \ref{teo:plane} holds true for any group
$G = \langle Z^{(k)}(G), \phi_{1}, \hdots, \phi_{p} \rangle$ of nilpotency class
$\leq k+1$ if $\mathrm{Fix}(\phi_{l})$ is compact for any
$1 \leq l \leq p$.

${\bf B_{k,p}}$: Theorem \ref{teo:plane} holds true for any group
$G = \langle Z^{(k)}(G), \phi_{1}, \hdots, \phi_{p} \rangle$ of nilpotency class
$\leq k+1$ if $\mathrm{Fix} \langle Z^{(k)}(G), \phi_{1}, \hdots, \phi_{p-1} \rangle$
is compact.

${\bf C_{k,p}}$: Theorem \ref{teo:plane} holds true for any group
$G = \langle Z^{(k)}(G), \phi_{1}, \hdots, \phi_{p} \rangle$ of nilpotency class
$\leq k+1$.

We can consider the hypotheses in Statements $A_{k,p}$, $B_{k,p}$, $C_{k,p}$
as conditions of strong, weak and no compactness respectively.
The validity of $C_{0,0}$ is obvious
whereas $C_{0,1}$ is a consequence of Brouwer translation theorem.
Moreover $C_{k,0}$ implies
$C_{k,1}$ for $k >0$ since $\langle Z^{(k)}(G), \phi \rangle$ is a group of nilpotency
class less than or equal to $k$ for any $\phi \in G$.
More precisely, the quotient group $Z^{(k)}(G) / Z^{(k-1)}(G)$ is the
center of $G/Z^{(k-1)}(G)$ by definition of $Z^{(k)}(G)$.
The group $\langle Z^{(k)}(G), \phi \rangle/ Z^{(k-1)}(G)$
is abelian because it is generated by the center of $G/Z^{(k-1)}(G)$ and another element.
As a consequence we obtain
$Z^{(k)} (\langle Z^{(k)}(G), \phi \rangle) = \langle Z^{(k)}(G), \phi \rangle$.
\subsection{Thurston normal form}
In this section we provide a Thurston decomposition for nilpotent groups of
$C^{1}$ diffeomorphisms.
\begin{pro}
\label{pro:tnf}
Let $G$ be a finitely generated
nilpotent subgroup of $\mathrm{Diff}_{+}^{1}(S)$ where $S$ is a finitely
punctured surface. Consider
normal subgroups $H_{1}$, $H_{2}$ of $G$
and $G$-invariant compact sets $K_{1}$, $K_{2}$ such that
$K_{1} \subset \mathrm{Fix}(H_{1})$,
$K_{2} \subset \mathrm{Fix}(H_{2}) \setminus  \mathrm{Fix}(H_{1})$
and $K_{1} \cap K_{2} =\emptyset$.
Suppose that one of the following properties holds true:
\begin{itemize}
\item $K_{2}$ is finite.
\item $G \in {\mathcal G}_{k,p+1}$, $H_{1} \in {\mathcal G}_{k,p}$, $K_{1} = \mathrm{Fix}(H_{1})$,
$\mathrm{Fix}(H_{1}) \cap \mathrm{Fix}(H_{2}) = \emptyset$
and $C_{k,p}$ holds true.
\end{itemize}
We define $M=S \setminus (K_{1} \cup K_{2})$. Suppose further $\chi(M)<0$ if
$K_{1} \cup K_{2}$ is a finite set.
There exists a finite family ${\mathcal R}$ of pairwise disjoint, pairwise non-isotopic rel $K_{1} \cup K_{2}$,
simple closed curves in $M$ and an union ${\mathcal A}$ of annular neighborhoods of the curves
in ${\mathcal R}$ such that for any $\phi \in G$ there exists $\theta \in \mathrm{Homeo}_{+}(S)$
isotopic to $\phi$ rel $K_{1} \cup K_{2}$ with $\theta ({\mathcal A}) = {\mathcal A}$.
Moreover, given any connected component $X$ of $S \setminus {\mathcal A}$ the construction satisfies:
\begin{enumerate}
\item If $(K_{1} \cup K_{2}) \cap X$ is infinite then either
$(K_{1} \cup K_{2}) \cap X = K_{1} \cap X$ and $\theta_{|X} \equiv Id$
for any $\phi \in H_{1}$ or
$(K_{1} \cup K_{2}) \cap X = K_{2} \cap X$ and $\theta_{|X} \equiv Id$
for any $\phi \in H_{2}$.
\item If $(K_{1} \cup K_{2}) \cap X$ is a finite set then
$\chi (X \setminus (K_{1} \cup K_{2})) <0$ and the group generated by
$\theta_{|X}$ for all $\phi$ in the stabilizer of $X$ modulo isotopy rel $K_{1} \cup K_{2}$
is composed by pseudo-Anosov and periodic elements.
More precisely, it is a finite or virtually cyclic nilpotent group.
\end{enumerate}
\end{pro}
The curves in ${\mathcal R}$ are called {\it reducing curves}.
This result is the analogue of Proposition 2.11 of \cite{FHPs} adapted to the
nilpotent setting. The proof is also similar. We sketch the proof,
following the presentation of \cite{FHPs} and stressing the
points where changes are required.
\begin{defi}
\label{def:dual}
We define the dual tree of ${\mathcal R}$. The vertices are the connected components of
$S \setminus {\mathcal A}$. We define an edge for any $\gamma \in {\mathcal R}$; it
connects the two connected components of $S \setminus {\mathcal A}$ that
share boundary components with the annulus neighborhood of $\gamma$.
If we assign length one to every edge, the group $G$ acts on the tree
by isometries.
\end{defi}
Later on we will apply the next result to the dual tree.
\begin{pro}
\label{pro:iso}
(cf. \cite[p. 20]{Ser-Sti})
Let $G$ be a group acting by isometries on a finite tree.
Then the action has a fixed point.
\end{pro}
Let $S$ be a finitely punctured surface. Consider a
finitely generated subgroup
${\mathcal F}$ of $\mathrm{Diff}_{+}^{1}(S)$.
Let $B$ and $C$ compact ${\mathcal F}$-invariant sets
such that
$B \subset \mathrm{Fix}({\mathcal F})$ and
$C \cap \mathrm{Fix}({\mathcal F}) = \emptyset$
Suppose that the image of ${\mathcal F}$ in the group of classes of
homeomorphisms modulo isotopy rel $B \cup C$ is nilpotent.
\begin{defi}
\label{FHPs}
Let ${\mathcal W}({\mathcal F}, B, C)$ be the set of compact subsurfaces
(they can contain punctures) $W \subset S$ such that:
\begin{itemize}
\item (${\mathcal W}-1$) $\partial W$ has finitely many components;
each of them is contained in
and is essential in $M := S \setminus (B \cup C)$;
\item (${\mathcal W}-2$) $B \setminus W$ is a finite set and every component of $W$
intersects $B$ in an infinite set;
\item (${\mathcal W}-3$) for all $f \in  {\mathcal F}$ the following is satisfied:
\begin{itemize}
\item (${\mathcal W}-3f$) there exists $\phi: S \to S$ isotopic to $f$ rel $B \cup C$ such that $\phi_{|W}$ is the
identity.
\end{itemize}
\end{itemize}
\end{defi}
The set ${\mathcal W}({\mathcal F}, B, C)$ is not empty if
$B$ is an infinite set. This is Lemma 2.5 of \cite{FHPs} for
the commutative case. The proof holds true for any
finitely generated subgroup of $\mathrm{Diff}_{+}^{1}(S)$.

Consider the setup in Proposition \ref{pro:tnf}.
The idea is that a maximal element $W$ of ${\mathcal W}({\mathcal F}, B, C)$
is unique up to isotopy. If we define
${\mathcal F}=H_{1}$, $B=K_{1}$ and $C=K_{2}$
the uniqueness implies that $W$
is invariant by $G$ up to isotopy rel $K_{1} \cup K_{2}$.
Naturally the connected components of $W$ are pieces of
the decomposition, they satisfy the condition (1) in Proposition \ref{pro:tnf}.

Let us show that ${\mathcal W}({\mathcal F}, B,C)$ has maximal elements.
We can define a partial order in ${\mathcal W}({\mathcal F}, B,C)$.
We say that $W_{1} < W_{2}$ if $W_{1}$ is isotopic rel $B \cup C$
to a subsurface of  $W_{2}$  but $W_{1}$ is not isotopic rel $B \cup C$
to $W_{2}$.
The next lemma is a version of
Lemma 2.6 of \cite{FHPs} adapted to the nilpotent setting.

\begin{lem}
\label{lem:existw}
Suppose that the set $B$ is infinite and that one of the following conditions
is satisfied:
\begin{itemize}
\item $C$ is finite.
\item ${\mathcal F}$ belongs to ${\mathcal G}_{k,p}$, $C_{k,p}$ holds true
and $B = \mathrm{Fix}({\mathcal F})$.
\end{itemize}
Then ${\mathcal W}({\mathcal F},B,C)$ contains maximal subsurfaces.
\end{lem}
\begin{proof}
The proof for the case $\sharp C < \infty$ is the same as in
Lemma 2.6 of \cite{FHPs}.
Moreover they prove that if the result is false there exists a disc
$D$ in $S$ that is ${\mathcal F}$-invariant
modulo isotopy rel $B \cup C$, disjoint from $B$ and such that
$C_{D}: = C \cap D$ is non-empty.

Let $X$ be the universal covering of the connected component of
$S \setminus (B \cup (C \setminus C_{D}))$ containing $D$.
We choose a lift $\tilde{D}$ of $D$ to $X$.
It is also a disc. Given $f \in {\mathcal F}$ there exists
a unique lift $\tilde{f}$ of $f$ to $X$ such that
$\tilde{f} (\tilde{C}_{D}) = \tilde{C}_{D}$
where $\tilde{C}_{D}$ is the lift of $C_{D}$ to $\tilde{D}$.
In this way we obtain a lift $\tilde{\mathcal F}$ of
${\mathcal F}$ to $X$ that preserves the compact set
$\tilde{C}_{D}$. Since $C_{k,p}$ holds true,
$\mathrm{Fix}(\tilde{\mathcal F})$ is non-empty.
Therefore $\mathrm{Fix}({\mathcal F}) \cap (S \setminus B)$
is non-empty, a contradiction.
\end{proof}
The Thurston decomposition for nilpotent groups is proved as a part of
the inductive step (we need to use that $C_{k,p}$ holds true).
This is not necessary in \cite{FHPs} since in the case $\sharp C = \infty$
the group ${\mathcal F}$ is always cyclic and the contradiction is provided by
the Brouwer translation theorem.

The set ${\mathcal W}({\mathcal F}, B, C)$ has a unique maximal element, i.e.
two maximal elements of ${\mathcal W}({\mathcal F}, B, C)$ are isotopic
rel $B \cup C$.
The proofs in Lemmas 2.7 and 2.8 of \cite{FHPs} apply here.
\begin{proof}[Proof of Proposition \ref{pro:tnf}]
We follow the steps of Proposition 2.11 of \cite{FHPs}. We include a proof
for the sake of clarity.

  We define $W_{1} = \emptyset$ if $K_{1}$ is finite, otherwise
consider a maximal element $W_{1}$ of
${\mathcal W}(H_{1},K_{1},K_{2})$.
Since $H_{1}$ is a normal subgroup of $G$ and $K_{1}$, $K_{2}$ are
$G$-invariant, we obtain that
$\phi(W_{1})$ is isotopic to $W_{1}$ rel $K_{1} \cup K_{2}$ for any
$\phi \in G$. There exists a diffeomorphism $\theta_{\phi} : S \to S$
such that $\phi$ is isotopic to $\theta_{\phi}$ rel $K_{1} \cup K_{2}$
and $\theta_{\phi}(W_{1}) = W_{1}$ for any $\phi \in G$.
Moreover we can suppose that $(\theta_{\phi})_{|W_{1}} \equiv Id$ if $\phi \in H_{1}$.

We denote the finitely punctured surface $S \setminus W_{1}$ by $S'$.
Consider a set of generators $\{ \phi_{1},\hdots, \phi_{m}\}$ of $G$
such that $H_{2}= \langle \phi_{1},\hdots, \phi_{n} \rangle$ for some $n \leq m$.
We can suppose that $\{ \phi_{1},\hdots, \phi_{m}\}$
and $\{\phi_{1},\hdots, \phi_{n}\}$ contain sets of generators of
$H_{1}$ and $H_{1} \cap H_{2}$ respectively.
We denote $G'= \langle \theta_{\phi_{1}}, \hdots, \theta_{\phi_{m}} \rangle$
and $H_{2}'= \langle \theta_{\phi_{1}}, \hdots, \theta_{\phi_{n}} \rangle$.
They are subgroups of $\mathrm{Diff}_{+}^{1}(S')$.
Moreover $H_{2}'$ is a normal subgroup of $G'$ modulo isotopy
rel $S' \cap (K_{1} \cup K_{2})$.
If $K_{2}$ is finite we define $W_{2}=\emptyset$.
Otherwise consider a maximal element $W_{2}$ of
${\mathcal W}(H_{2}',K_{2},K_{1} \setminus W_{1})$.
We can suppose that $(\theta_{\phi_{j}})_{|W_{2}} \equiv Id$ for any $1 \leq j \leq n$
up to an isotopy rel $K_{1} \cup K_{2}$.
The maximal nature of $W_{2}$ implies that up to an isotopy
rel $K_{1} \cup K_{2}$ the diffeomorphism
$\theta_{\phi_{j}}$ preserves $W_{2}$ for any $1 \leq j \leq m$.

The connected components of $\partial W_{1} \cup \partial W_{2}$
are natural candidates to be reducing curves.
Anyway we can have redundant curves. More precisely,
if two curves in $\partial W_{1} \cup \partial W_{2}$ are the two
boundary components of an unpunctured annulus contained in $S$
then we replace them with a core curve of the annulus.
In this way we obtain a set ${\mathcal R}'$ of reducing curves.
It is easy to tweak slightly the construction so that
$\theta_{\phi}$ preserves an annular neighborhood
${\mathcal A}$ of ${\mathcal R}'$ for any $\phi \in G$
and that condition (1) in Proposition \ref{pro:tnf}
still holds true.
We define $S''=S \setminus (W_{1} \cup W_{2})$, it is a finite
type subsurface of $S$.
Let $G_{0}$ be the group induced by $G$ or $G'$ in the mapping
class group of $S''$.
The group $G_{0}$ is nilpotent, moreover it is virtually abelian
(Theorem B of \cite{BLM}).
The reduction of $G_{0}$ to Thurston normal form is equivalent to the reduction
of a finite index normal subgroup.
Indeed a homeomorphism is irreducible (i.e. periodic or pseudo-Anosov) if and only
if an iterate is irreducible.
So in order to reduce $G_{0}$ it suffices to reduce an abelian subgroup of
a mapping class group of a finitely punctured surface. This is well known,
see for example Lemma 2.2 of \cite{Handel-top}.
Lemma \ref{lem:model} implies that we can complete
${\mathcal R}'$ to obtain a set of reducing curves ${\mathcal R}$
such that condition (2) holds true whereas we still preserve
the other properties.
\end{proof}
\subsection{Proof of Theorem \ref{teo:plane}}
Our point of view is replacing a subgroup $G$ of $\mathrm{Diff}_{+}^{1}({\mathbb R}^{2})$
with an irreducible nilpotent model that defines the same image in the
mapping class group of a subsurface of ${\mathbb R}^{2}$, obtaining
information for the model and then interpreting it for the initial group.
One of the reasons because this approach works is that nilpotent groups
are quite rigid. For instance the next lemma implies that in
Theorem \ref{teo:plane} we can replace the condition of existence of
a $G$-invariant non-empty compact set with the existence of an element of
$G$ whose fixed point set is a non-empty compact set.
A condition on a single element implies the existence of a global fixed point for the group.
\begin{lem}
\label{lem:upt}
Let $G = \langle Z^{(k)}(G), \phi_{1}, \hdots, \phi_{p} \rangle$ be a finitely generated
nilpotent subgroup of
$\mathrm{Diff}_{+}^{1}({\mathbb R}^{2})$ of nilpotency class $\leq k+1$.
Suppose that there exists
$\alpha \in G$ such that $\mathrm{Fix}(\alpha)$ is a non-empty compact set.
Suppose further that $C_{k,p}$ holds true. Then
$\mathrm{Fix} (G)$ is non-empty.
\end{lem}
\begin{proof}
The series
\[ 1 \leq \ \langle Z^{(0)} (G),{\alpha} \rangle \ \leq \ \langle Z^{(1)} (G),{\alpha} \rangle \ \leq
\hdots \leq \ \langle Z^{(k)} (G),{\alpha} \rangle \]
is subnormal and every quotient of consecutive subgroups
is abelian. Let us prove that
$\mathrm{Fix} \langle Z^{(l)} (G),\alpha \rangle$ is non-empty for any $0 \leq l \leq k$.
It is obvious for $l=0$ and if $l<k$ the compact set
$\mathrm{Fix} \langle Z^{(l)} (G),\alpha \rangle$ is invariant by
$\langle Z^{(l+1)} (G),\alpha \rangle$.
We obtain $\mathrm{Fix} \langle Z^{(l+1)} (G),\alpha \rangle \neq \emptyset$ by Statement
$C_{k,0}$. Since $\langle Z^{(k)} (G),\alpha \rangle$ is normal in
$G$, Statement $C_{k,p}$ implies that
$\mathrm{Fix} \langle Z^{(k)}(G), {\phi}_{1}, \hdots, {\phi}_{p} \rangle$ is not empty.
\end{proof}
\begin{lem}
\label{lem:aux1}
Let $G \in {\mathcal G}_{k,p+1}$ be a finitely generated subgroup of $\mathrm{Diff}_{+}^{1}({\mathbb R}^{2})$.
Consider a normal subgroup $H$ of $G$ and $f \in G$ such that
the class of $f$ in $G/H$ belongs to the center of $G/H$. We denote $L= \langle H,f \rangle$.
Suppose that $H \in {\mathcal G}_{k,p}$.
Let $K \subset \mathrm{Fix}(H)$ be a non-empty compact
$G$-invariant subset. Suppose that $C_{k,p}$ holds true.
Then one of the following possibilities is satisfied:
\begin{enumerate}
\item $K \cap \mathrm{Fix}(L) \neq \emptyset$.
\item There exists a lift $\tilde{G}$ of $G$ such that
      $\tilde{G} = \langle Z^{(k)}(\tilde{G}),\tilde{\alpha}_{1},\hdots,\tilde{\alpha}_{p+1} \rangle$
      and $\mathrm{Fix}(\tilde{\alpha}_{l})$ is a non-empty compact set for any
      $1 \leq l \leq p+1$.
      Moreover $\mathrm{Fix}(\tilde{f})$ is a non-empty compact set if $G=L$.
\item There exists a lift $\tilde{L}$ of $L$
      such that $\mathrm{Fix}(\tilde{f})$ is non-empty and compact.
      Moreover if $L \in  {\mathcal G}_{k,p}$ then
      $\cap_{\phi \in G} \phi (P(K, f) \cap \mathrm{Fix}(H))$ is a non-empty
      compact $G$-invariant set.
\end{enumerate}
The lift $\tilde{G}$ (resp. $\tilde{L}$) of $G$ (resp. $L$) is defined in the universal covering
of a connected component of ${\mathbb R}^{2} \setminus K'$ where $K'$ is some $G$-invariant
(resp. $L$-invariant) non-empty compact set contained in $K$.
\end{lem}
Let $G= \langle Z^{(k)}(G), {\phi}_{1}, \hdots, {\phi}_{p+1} \rangle$ in ${\mathcal G}_{k,p+1}$.
Consider $a_{1} < a_{2} < \hdots < a_{k+1}$ and a system of generators
$\{g_{1},\hdots,g_{a_{k+1}}\}$ of $G$ such that
$Z^{(l)}(G) = \langle g_{1},\hdots,g_{a_{l}} \rangle$ for any $1 \leq l \leq k+1$,
$a_{k+1}-a_{k}=p+1$ and
$g_{a_{k}+1}=\phi_{1}$, $\hdots$, $g_{a_{k+1}}=\phi_{p+1}$.
We denote $G_{0}=\{Id\}$ and $G_{j}= \langle g_{1},\hdots,g_{j} \rangle$. The group
$G_{j}$ is normal in $G$ for any $0 \leq j < a_{k+1}$.
We define $J_{j}=G_{j}$ for $j<a_{k+1}-1$ and $J_{a_{k+1}-1}=G_{a_{k+1}-2}$ for
$j=a_{k+1}-1$.
\begin{lem}
\label{lem:aux2}
Let $K \subset \mathrm{Fix}(G_{j})$ be a non-empty compact
$G$-invariant subset for $j<a_{k+1}$. Suppose that $C_{k,p}$ holds true.
Then one of the following possibilities is satisfied:
\begin{enumerate}
\item $K \cap \mathrm{Fix}(G_{j+1}) \neq \emptyset$.
\item There exists a lift $\tilde{G}$ of $G$ such that
      $\tilde{G} = \langle Z^{(k)}(\tilde{G}),\tilde{\alpha}_{1},\hdots,\tilde{\alpha}_{p+1} \rangle$
      and $\mathrm{Fix}(\tilde{\alpha}_{l})$ is a non-empty compact set for any
      $1 \leq l \leq p+1$.
\item $\cap_{\phi \in G} \phi (P(K, g_{j+1}) \cap \mathrm{Fix}(J_{j}))$ is a non-empty
      compact $G$-invariant set. Moreover if $j = a_{k+1}-1$ then there exists a lift
      $\tilde{G}$ of $G$
      such that $\mathrm{Fix}(\tilde{\phi}_{p+1})$ is non-empty and compact.
\end{enumerate}
\end{lem}
Lemma \ref{lem:aux2} justifies the introduction of Statements $A_{k,p}$ and $B_{k,p}$.
In Cases (1) and (3) for $j < a_{k+1}-1$ there exists a compact non-empty
$G$-invariant set $F$ contained in $\mathrm{Fix}(G_{j+1})$.
We can continue the process by replacing
$K$ with $F$ and $G_{j}$ with $G_{j+1}$.
Otherwise either (2) holds true and $\tilde{G}$ satisfies the hypothesis in
Statement $A_{k,p+1}$ or (3) holds true for $j=a_{k+1}-1$ and $\tilde{G}$
satisfies the hypothesis in
Statement $B_{k,p+1}$. Lemma \ref{lem:aux1} implies
Lemma \ref{lem:aux2} and it is a little bit more general.
\begin{proof}[Proof of Lemma \ref{lem:aux1}]
We can suppose that $K \cap \mathrm{Fix}(L) = \emptyset$.
Consider a compact set $K^{1} \subset K$ that is minimal by the action of
$G$. We apply Proposition \ref{pro:tnf} to
$H_{1}=H$, $H_{2}=G$, $K_{1}=K^{1}$, $K_{2}=\emptyset$.
Consider the tree $\Gamma$ that is dual to ${\mathcal R}$ (cf. Definition \ref{def:dual}).
Then $G$ fixes the vertex that corresponds to the connected
component $D_{\infty}$ of $\infty$.
Consider the subtree $\Gamma'$ of $\Gamma$ of points fixed by $G$.
If $\Gamma' = \{D_{\infty}\}$ we define $D=D_{\infty}$.
Otherwise
there exists a vertex $D \neq D_{\infty}$ of $\Gamma'$ that belongs
to exactly one edge $\gamma$ of $\Gamma'$.

Suppose that $D \cap K^{1}$ is finite.
Let $\hat{D}$ be the marked sphere obtained from $\overline{D}$
by contracting $\overline{D} \cap {\mathcal A}$.
Consider $[]: G \to \mathrm{MCG}(G_{|\hat{D}})$. The
model for $[G]$ provided by Lemma \ref{lem:model}
has exactly a fixed point at $\overline{D} \cap ({\mathcal A} \cup \{\infty\})$
and it does not fix any other
puncture or curve in ${\mathcal A}$.
Hence it fixes an interior point by Remark \ref{rem:pt2}.
There exists a lift $\tilde{G}$ of $G$ to the
universal covering of
the connected component of
${\mathbb R}^{2} \setminus K^{1}$ containing $D \setminus K^{1}$
such that
$\tilde{G} = \langle Z^{(k)} (\tilde{G}),\tilde{\alpha}_{1},\hdots,\tilde{\alpha}_{p+1} \rangle$ and
$\mathrm{Fix}(\tilde{\alpha}_{l})$ is a non-empty compact set for any
$1 \leq l \leq p+1$ (Proposition \ref{pro:lif2}).
If $G=L$ the actions of $[G]$ and $[f]$ on punctures and reducing curves coincide.
Hence $\mathrm{Fix}(\tilde{f})$ is a non-empty compact set by Proposition \ref{pro:lif2}.

Suppose $D \cap K^{1}$ is an infinite set. The minimality implies
$K^{1} \subset D$ and ${\mathcal R} = \emptyset$.
In particular all the elements in $H$ are isotopic to the identity
rel $K^{1} \cup \{ \infty\}$.

Consider a compact set $K^{2} \subset K^{1}$ that is minimal by the action of
$L$.
Let $E_{\infty}$ be the unbounded connected component of
${\mathbb R}^{2} \setminus K^{2}$.
Consider the full compact set
$\hat{K}^{2} := {\mathbb R}^{2} \setminus E_{\infty}$.
Suppose that $\mathrm{Fix}(f) \cap \hat{K}^{2}$
is non-empty. Then there exists a bounded connected
component $D_{1}$ of ${\mathbb R}^{2} \setminus K^{2}$
such that $\mathrm{Fix}(f) \cap D_{1} \neq \emptyset$.
Moreover the minimality implies
$K^{2} \subset \overline{E_{\infty}}$, hence
$D_{1}$ is a topological disc.
The diffeomorphism $f$ has no fixed points in $\partial D_{1}$;
thus the set $\mathrm{Fix}(f_{|D_{1}})$ is compact.
Since $D_{1}$ is $L$-invariant, we
define $\tilde{L}$ as the group obtained by
restricting $L$ to $D_{1}$.
Moreover if $L \in  {\mathcal G}_{k,p}$ we
obtain that
$\mathrm{Fix}(L) \cap D_{1}$ is non-empty by Lemma \ref{lem:upt}.
Clearly $\mathrm{Fix}(L) \cap D_{1}$ and all its iterates by
elements of $G$ are contained in
$P(K^{1},f) \cap \mathrm{Fix}(H)$.

We suppose $\mathrm{Fix}(f) \cap \hat{K}^{2}= \emptyset$
from now on.
The set $P({K}^{2},f)$ is equal to $P(\hat{K}^{2},f)$, it is
a non-empty subset of $E_{\infty}$ by Proposition \ref{pro:niel}.
Since $\hat{K}^{2}$ is a full compact set, it is not connected; otherwise
we obtain $P(\hat{K}^{2},f) = \emptyset$.
In particular $K^{2}$ is not connected.

Let $\sigma: \tilde{E}_{\infty} \to E_{\infty}$ be the universal covering of
$E_{\infty}$.
Given $\phi \in H$ let ${\{\beta_{t}\}}_{t \in [0,1]}$
be an isotopy rel $K^{2}$ from
$\beta_{0}=Id$ to $\beta_{1}=\phi$.
Consider the lift ${\{\tilde{\beta}_{t}\}}_{t \in [0,1]}$
to $\tilde{E}_{\infty}$ such that $\tilde{\beta}_{0} =Id$.
We define $\tilde{\phi}=\tilde{\beta}_{1}$.
It is the identity lift of $\phi \in H$ to $\tilde{E}_{\infty}$;
it is uniquely defined
since the Euler characteristic of $E_{\infty}$ is negative.
We define
$\mathrm{Fix}_{c}(\phi)= \sigma (\mathrm{Fix}(\tilde{\phi}))$ and
$\mathrm{Fix}_{c}(H)= \cap_{\phi \in H} \mathrm{Fix}_{c}(\phi)$.
Let us prove that
$P(K^{2},f) \cap \mathrm{Fix}_{c}(H)$ is contained in
$\phi (P(K, f) \cap \mathrm{Fix}(H))$ for any $\phi \in G$.
Given $\phi \in G$ we have that $\alpha: = \phi f \phi^{-1} f^{-1}$ belongs to
$H$ by hypothesis.
We claim
\begin{equation}
\label{equ:con}
P(K^{2}, \alpha f) \cap \mathrm{Fix}_{c}(H) = P(K^{2}, f) \cap \mathrm{Fix}_{c}(H) .
\end{equation}
Indeed if an element $z$ of $\mathrm{Fix}_{c}(H) \cap \mathrm{Fix}(f)$ does not belong to
$P(K^{2}, \alpha f)$ then there exists a path $\rho$ from $z$ to $\infty$
in ${\mathbb S}^{2} \setminus K^{2}$
such that $\rho$ and $(\alpha f)(\rho)$ are homotopic relative to ends in
${\mathbb S}^{2} \setminus K^{2}$. Let ${\{ \alpha_{t} \}}_{t \in [0,1]}$ be
an isotopy rel $K^{1} \cup \{\infty\}$ such that $\alpha_{0}=Id$ and $\alpha_{1}=\alpha$.
We obtain that $\rho$ and
the path $\varsigma f(\rho)$ are homotopic relative to their ends
in ${\mathbb S}^{2} \setminus K^{2}$ where
$\varsigma:[0,1] \to {\mathbb R}^{2} \setminus K^{1}$ is defined by
$\varsigma(t)=\alpha_{t}^{-1}(z)$.
The class of $\varsigma$ is trivial in $\pi_{1}({\mathbb R}^{2} \setminus K^{2},z)$
since $z \in \mathrm{Fix}_{c}(\alpha)$.
Thus $z$ does not belong to $P(K^{2}, f)$.
Analogously we prove that $z  \in (\mathrm{Fix}_{c}(H) \cap  \mathrm{Fix}(f)) \setminus P(K^{2},f)$
implies $z \not \in P(K^{2}, \alpha f)$.

We deduce
\[   P(K^{2}, f) \cap \mathrm{Fix}_{c}(H) \subset P(\phi(K), \alpha f) \cap \mathrm{Fix}(H)
=\]
\[
P(\phi(K), \phi f \phi^{-1}) \cap \mathrm{Fix}(H) = \phi (P(K, f) \cap \mathrm{Fix}(H)) . \]
The inclusion is a consequence of equation (\ref{equ:con}) and $K^{2} \subset K$.
In order to prove that
$\cap_{\phi \in G} \phi (P(K, f) \cap \mathrm{Fix}(H))$ is non-empty it suffices to show
that $P(K^{2}, f) \cap \mathrm{Fix}_{c}(H)$ is non-empty.


We apply the Thurston decomposition
(Proposition \ref{pro:tnf}) to the group $L$ and
the marked surface $({\mathbb S}^{2},\{\infty\})$
($H_{1}=H$, $H_{2}=L$, $K_{1}=K^{2}$, $K_{2}=\emptyset$).
Let $D_{\infty}'$ be the connected component of
${\mathbb S}^{2} \setminus {\mathcal A}$ that contains $\infty$.
Suppose that $D_{\infty}' \cap K^{2}$ is a finite set. The minimality
implies that $f$ can not fix any curve in
$\overline{D_{\infty}'} \cap {\mathcal A}$.
Analogously as in the case $\sharp (D \cap K^{1}) < \infty$
Proposition \ref{pro:lif2} implies the existence of
a lift $\tilde{L}$ of $L$ to $\tilde{E}_{\infty}$
such that $\mathrm{Fix}(\tilde{f})$
is a non-empty compact set. Moreover $\sigma(\mathrm{Fix}(\tilde{f}))$
is contained in $P(K^{2},f)$.
We obtain $\mathrm{Fix}(\tilde{L}) \neq \emptyset$
in the case $L \in {\mathcal G}_{k,p}$ by Lemma \ref{lem:upt}.
The set $P(K^{2},f) \cap \mathrm{Fix}_{c}(H)$ is not empty
as it contains
$\sigma(\mathrm{Fix}(\tilde{L}))$.

Suppose that $D_{\infty}' \cap K^{2}$ is an infinite set.
In this case $K^{2}$ is contained in $D_{\infty}'$ by
minimality and ${\mathcal R}$ is empty.
We can suppose that $f$ does not preserve any essential
type subsurface $N \subset {\mathbb R}^{2} \setminus K^{2}$
modulo isotopy rel $K^{2}$.
Otherwise we replace ${\mathcal R}$ with $\partial N$.
The unbounded connected component of
${\mathbb R}^{2} \setminus \partial N$ does not contain any point of $K^{2}$
by minimality. We can argue as in the finite type case.

Let $z \in P(K^{2},f)$. Consider a lift $\tilde{z} \in \tilde{E}_{\infty}$
of $z$. We define $\tilde{f}$ as the lift
of $f$ to $\tilde{E}_{\infty}$ such that
$\tilde{f}(\tilde{z})=\tilde{z}$.
The map $f$ does not preserve any essential non-peripheral homotopy
class in
$\pi_{1}(E_{\infty},z)$. Otherwise $f$ preserves an essential finite
type subsurface $N \subset {\mathbb R}^{2} \setminus K^{2}$
modulo isotopy rel $K^{2}$ by Lemma
2.12 of \cite{FHPs}.
Moreover $f$ does not preserve any other element of  $\pi_{1}(E_{\infty},z)$
since $z \in P(K^{2},f)$.
Hence $\mathrm{Fix}(\tilde{f})$ is a non-empty compact set.
No iterate of $f$ belongs to $H$, otherwise $K^{2}$ is finite.
The mapping $\tau: L \to \langle f \rangle$ defined by
$\tau_{|H} \equiv Id$,   $\tau(f)=f$ is a well-defined homomorphism
of groups.
All the elements of $H$ are isotopic to the identity
rel $K^{2}$.
Given $\phi \in L$ let ${\{\beta_{t}\}}_{t \in [0,1]}$
be an isotopy rel $K^{2}$ from
$\beta_{0}=\tau(\phi)$ to $\beta_{1}=\phi$.
Consider the lift ${\{\tilde{\beta}_{t}\}}_{t \in [0,1]}$
to $\tilde{E}_{\infty}$ such that $\tilde{\beta}_{0}(\tilde{z})=\tilde{z}$.
We define $\tilde{\phi}=\tilde{\beta}_{1}$.
The group $\tilde{L}$
is a lift of $L$ to $\tilde{E}_{\infty}$ by the uniqueness of the construction.
Moreover $\tilde{H}$ is the identity lift
of $H$. Suppose that $L \in {\mathcal G}_{k,p}$.
Lemma \ref{lem:upt} implies
$\mathrm{Fix}(\tilde{L}) \neq \emptyset$.
The Nielsen class of the elements of $\sigma (\mathrm{Fix}(\tilde{f}))$
is the Nielsen class of $z$.
Thus $P(K^{2},f) \cap \mathrm{Fix}_{c}(H)$ is non-empty.
\end{proof}
\begin{proof}[Proof of Lemma \ref{lem:aux2}]
Lemma \ref{lem:aux1} for $H=G_{j}$ and $f=g_{j+1}$ implies all the properties except that
$\cap_{\phi \in G} \phi (P(K, g_{a_{k+1}}) \cap \mathrm{Fix}(J_{a_{k+1}}))$ is a non-empty
compact $G$-invariant set if $j=a_{k+1}-1$  and neither Case (1) nor Case (2)
holds.
In such a situation the result is a consequence of Lemma \ref{lem:aux1}
applied to $H=J_{a_{k+1}-1}$ and $f=g_{a_{k+1}}$
since $K \cap \mathrm{Fix}(g_{a_{k+1}}) = \emptyset$ and Case (2)
of Lemma \ref{lem:aux1} does not hold.
\end{proof}
Let us prove next $C_{k,p} \implies A_{k,p+1}$ and $C_{k,p} \implies B_{k,p+1}$.
The proof of $C_{k,p} \implies C_{k,p+1}$ is an easy consequence of these
results and Lemma \ref{lem:aux2}.
\begin{lem}
\label{lem:cia}
Statement $C_{k,p}$ implies Statement $A_{k,p+1}$ for all
$k \geq 0$, $p \geq 0$.
\end{lem}
\begin{proof}
Denote
\[ H_{1} =   \langle Z^{(k)}(G), \phi_{1}, \hdots, \phi_{p} \rangle , \ \
H_{2} =   \langle Z^{(k)}(G), \phi_{1}, \hdots, \phi_{p-1}, \phi_{p+1} \rangle. \]
We define $K_{1} = \mathrm{Fix}(H_{1})$ and $K_{2} = \mathrm{Fix}(H_{2})$.

The hypothesis implies that $K_{1}$ and $K_{2}$ are compact. Moreover
$C_{k,p}$ implies that none of them is empty.
It suffices to prove that $K_{1} \cap K_{2} = \emptyset$
leads to a contradiction.

We apply Proposition \ref{pro:tnf} to
$H_{1}$, $H_{2}$.
Consider the tree $\Gamma$ that is dual to ${\mathcal R}$.
We define $D$ as in Lemma \ref{lem:aux2}.
It is a $G$-invariant (modulo isotopy rel $K_{1} \cup K_{2}$) component of
${\mathbb S}^{2} \setminus {\mathcal A}$
such that there is exactly one point or curve $\gamma$ fixed by
$G$ in the union of $\{\infty\} \cup (\overline{D} \cap {\mathcal A})$
modulo isotopy rel $K_{1} \cup K_{2}$.
If $D=D_{\infty}$ we define $D'={\mathbb R}^{2}$.
Otherwise
we define $D'$ as the connected component of ${\mathbb S}^{2} \setminus \gamma$
not containing $\infty$. We have $D \subset D'$.

Suppose $K_{1} \cap D' \neq \emptyset$ and  $K_{2} \cap D' \neq \emptyset$.
If $\sharp (K_{1} \cap D) = \infty$
then the elements in $H_{1}$ fix the curves in $\overline{D} \cap {\mathcal A}$
modulo isotopy rel $K_{1} \cup K_{2}$.
Since $\emptyset \neq K_{2} \cap D' \subset \mathrm{Fix}(H_{2})$, there exists a
curve in $\overline{D} \cap {\mathcal A}$
different than $\gamma$
and invariant by the group
$G= \langle H_{1},H_{2} \rangle$. This is impossible by the choice of $D$.
Hence $\sharp (K_{1} \cap D)=\infty$ implies
$K_{2} \cap D' = \emptyset$. Analogously $\sharp (K_{2} \cap D)=\infty$ implies
$K_{1} \cap D' = \emptyset$.
Either $(K_{1} \cup K_{2}) \cap D$ is finite or one of the sets
$K_{1} \cap D'$, $K_{2} \cap D'$ is empty.

Suppose that $(K_{1} \cup K_{2}) \cap D$ is a finite set.
Let $\hat{D}$ be the marked sphere obtained from $\overline{D}$
by contracting the curves in $\overline{D} \cap {\mathcal A}$.
Consider $[]: G \to \mathrm{MCG}(G_{|\hat{D}})$. The
model for $[G]$ provided by Lemma \ref{lem:model}
has exactly a fixed point at $\overline{D} \cap ({\mathcal A} \cup \{\infty\})$
and it does not fix any other
puncture or curve in ${\mathcal A}$.
Hence it fixes an interior point (Remark \ref{rem:pt2}).
Proposition \ref{pro:lif2}
implies that there exists a lift $\tilde{G}$ of $G$ to the
universal covering of the connected component of ${\mathbb R}^{2} \setminus (K_{1} \cup K_{2})$
containing $D \setminus (K_{1} \cup K_{2})$
such that
$\tilde{G} = \langle Z^{(k)} (\tilde{G}),\tilde{\alpha}_{1},\hdots,\tilde{\alpha}_{p+1} \rangle$ and
$\mathrm{Fix}(\tilde{\alpha}_{j})$ is a non-empty compact set for any
$1 \leq j \leq p+1$.
Lemma \ref{lem:upt} implies that $\mathrm{Fix} \langle Z^{(k)} (\tilde{G}),\tilde{\alpha}_{1} \rangle$
is a non-empty $\tilde{G}$-invariant compact set.
We obtain
$\mathrm{Fix} \langle Z^{(k)}(\tilde{G}), \tilde{\phi}_{1}, \hdots, \tilde{\phi}_{p} \rangle \neq \emptyset$
by Statement $C_{k,p}$. We have
\[ \emptyset \neq \mathrm{Fix} \langle Z^{(k)}(G), {\phi}_{1}, \hdots, {\phi}_{p} \rangle \cap
({\mathbb R}^{2} \setminus (K_{1} \cup K_{2})) = \emptyset . \]
It represents a contradiction.

Suppose $\sharp ((K_{1} \cup K_{2}) \cap D) = \infty$.
Then either $\sharp (K_{1} \cap D) = \infty$ and $K_{2} \cap D'= \emptyset$
or $\sharp (K_{2} \cap D) = \infty$ and $K_{1} \cap D'= \emptyset$.
Suppose that we are in the latter case without lack of generality.
Clearly we have $D \neq D_{\infty}$. Consider a lift
$\tilde{D}'$ of the disc $D'$ to the universal covering
$M$ of the connected component of
${\mathbb R}^{2} \setminus ((K_{1} \cup K_{2}) \setminus D')$ containing $D'$.
Let $\tilde{K}$ be the lift of $(K_{1} \cup K_{2}) \cap D'$ in $\tilde{D}'$.
There exists a unique lift $\tilde{G}$ of $G$ to $M$
such that $\tilde{\phi}(\tilde{K}) = \tilde{K}$ for any $\phi \in G$.
The group $\tilde{G}$ preserves the compact set $\tilde{K}$;
we obtain
$\mathrm{Fix} (\tilde{H}_{1}) \neq \emptyset$ and
$\mathrm{Fix} (\tilde{H}_{2}) \neq \emptyset$
by Statement $C_{k,p}$. We deduce
that the intersection of
$\mathrm{Fix} ({H}_{1})$
and ${\mathbb R}^{2} \setminus K_{1}$ is not empty; this is absurd.
\end{proof}
\begin{lem}
\label{lem:cib}
Statement $C_{k,p}$ implies Statement $B_{k,p+1}$ for all
$k \geq 0$, $p \geq 0$.
\end{lem}
\begin{proof}
We define $H_{1}$, $H_{2}$ and $K_{1}$ as in Lemma \ref{lem:cia}.
The set $K_{1}$ is compact by hypothesis.
We apply Lemma \ref{lem:aux2}.
If $K_{1} \cap \mathrm{Fix}(G)$ is not empty there is nothing to prove.
If the second Case is satisfied then Lemma \ref{lem:cia} implies
$\mathrm{Fix}(G) \neq \emptyset$.
So we can suppose that
\[ K_{2}:=
\cap_{\phi \in G} \phi (P(K_{1}, \phi_{p+1}) \cap \mathrm{Fix} \langle Z^{(k)}(G), \phi_{1}, \hdots, \phi_{p-1} \rangle) \]
is non-empty. We apply Proposition \ref{pro:tnf} to
$H_{1}$, $H_{2}$.

The remainder of the proof is exactly the same as in Lemma \ref{lem:cia} except for the
case when $\sharp (K_{1} \cap D) = \infty$ and $K_{2} \cap D'= \emptyset$.
We apply Lemma \ref{lem:aux2} to the compact $G$-invariant set
$K_{1} \cap D'$. If we are in the second Case we obtain
$\mathrm{Fix}(G) \neq \emptyset$ by
applying Lemma \ref{lem:cia}. We can suppose that
\[ K_{2}':= \cap_{\phi \in G} \phi \left(
P(K_{1} \cap D', {\phi}_{p+1}) \cap
\mathrm{Fix} \langle Z^{(k)}(G), {\phi}_{1},\hdots,{\phi}_{p-1} \rangle
\right) \]
is a non-empty compact $G$-invariant set contained in $K_{2}$.
We obtain
\[ K_{2}' = (K_{2}' \cap D') \cup (K_{2}' \setminus \overline{D'}) =
K_{2}' \setminus \overline{D'} \subset {\mathbb R}^{2} \setminus P(K_{1} \cap D', {\phi}_{p+1}) \]
since $\partial D'$ is invariant by $\phi_{p+1}$ modulo an isotopy
rel $K_{1} \cup K_{2}$.
This contradicts the property $K_{2}' \subset P(K_{1} \cap D', {\phi}_{p+1})$.
\end{proof}
%
%
%
\begin{lem}
\label{lem:cic}
Statement $C_{k,p}$ implies Statement $C_{k,p+1}$ for all
$k \geq 0$, $p \geq 0$.
\end{lem}
\begin{proof}
Consider the notations above Lemma \ref{lem:aux2}.
Let $K$ be a compact non-empty $G$-invariant set.
We define
\[ j = \max \{ l \in \{0,\hdots,a_{k+1}\}: K \cap \mathrm{Fix}(G_{j}) \neq \emptyset \}. \]
If $j=a_{k+1}$ we are done.
We apply Lemma \ref{lem:aux2} to $K \cap \mathrm{Fix}(G_{j})$.
The first Case is impossible by the choice of $j$.
If the second Case is satisfied then
$\mathrm{Fix}(G)$ is non-empty by Lemma \ref{lem:cia}.
So we can suppose that we are in the third Case.
If $j < a_{k+1}-1$ then the set
$\cap_{\phi \in G} \phi (P(K, g_{j+1}) \cap \mathrm{Fix}(G_{j+1}))$
is a non-empty compact $G$-invariant set contained in
$\mathrm{Fix}(G_{j+1})$.
Up to repeat the argument a finite number of times we obtain
a non-empty compact $G$-invariant set $K_{1}$
contained in $\mathrm{Fix} \langle Z^{(k)}(G),\phi_{1},\hdots,\phi_{p} \rangle$.
We apply Lemma \ref{lem:aux2} to $K_{1}$.
Either $G$ has a global fixed point or the last Case is satisfied.
In the latter situation there exists a lift $\tilde{G}$ of $G$ such that
$\mathrm{Fix}(\tilde{\phi}_{p+1})$ is compact.
Lemma \ref{lem:cib} implies that $\tilde{G}$ (and then $G$) has a global fixed
point.
\end{proof}
\begin{rem}
Notice that since we already proved Theorem \ref{teo:plane}, the condition
requiring $C_{k,p}$ to be satisfied in Proposition \ref{pro:tnf} and
Lemmas \ref{lem:existw}, \ref{lem:upt}, \ref{lem:aux1}  and \ref{lem:aux2} is superfluous.
\end{rem}
\begin{proof}[Proof of Theorem \ref{teo:plane2}]
Let ${\mathcal L}$ be the set of finitely generated subgroups of $G$ containing $\phi$.
Since Theorem \ref{teo:plane} and Lemma \ref{lem:upt}
hold true, the result is obvious for any
$H \in {\mathcal L}$. The set $\mathrm{Fix}(H)$ is
compact since it is a closed subset of $\mathrm{Fix}(\phi)$.
Any finite intersection
$\mathrm{Fix}(H_{1}) \cap \hdots \cap \mathrm{Fix}(H_{m})$
of sets in the family ${\{ \mathrm{Fix}(H)\}}_{H \in {\mathcal L}}$
is non-empty since
$\langle H_{1},\hdots,H_{m} \rangle$ belongs to ${\mathcal L}$.
Therefore $\mathrm{Fix}(G)= \cap_{H \in {\mathcal L}} \mathrm{Fix}(H)$ is a non-empty
set.
\end{proof}
\begin{proof}[Proof of Corollary \ref{cor:d2}]
Consider the group ${\mathcal F}$ of restrictions to
${\mathbb S}^{1} = \partial {\mathbb D}^{2}$ of elements of $G$.
Since ${\mathcal F}$ is nilpotent, there exists
a ${\mathcal F}$-invariant probability measure $\mu$.
Given an element $f \in  {\mathcal F}$ such that $\mathrm{Fix} (f) \neq \emptyset$
the dynamics of $f$ in each interval of ${\mathbb S}^{1} \setminus \mathrm{Fix} (f)$
is conjugate to a translation. We obtain $\mathrm{supp}(\mu) \subset \mathrm{Fix} (f)$.
Thus either $\mathrm{supp}(\mu) \subset \mathrm{Fix} ({\mathcal F})$ and
there is nothing to prove or there exists $\phi \in G$ such that
$\mathrm{Fix} (\phi) \cap \partial {\mathbb D}^{2} = \emptyset$.
In the latter case $\mathrm{Fix} (\phi)$ is compact in
${\mathbb D}^{2} \setminus \partial {\mathbb D}^{2}$.
Thus $G$ has a global fixed point by Theorem \ref{teo:plane2}.
\end{proof}
\begin{pro}
Consider the setup in Proposition \ref{pro:lif2}.
Let ${\mathcal F}$ be a nilpotent subgroup of
$\mathrm{Diff}_{+}^{1}(S)$ such that
$\eta({\mathcal F})$ is a non-trivial irreducible group.
Consider the model $\Theta$ of $\eta({\mathcal F})$
provided by Lemma \ref{lem:model}.
Let $z$ be a fixed point of $\Theta$ in the interior of $M_{1}$.
Then there exists $x \in \mathrm{Fix}({\mathcal F}) \cap M$
such that the $\phi$-Nielsen class of $x$ rel $K$ is
non-empty, compactly covered and does not peripherally contain
any punctures for any $\phi \in {\mathcal F}$ with $\eta(\phi) \neq Id$.
  Moreover the $\phi$-Nielsen class of $x$ rel $K$
is determined by the $\theta_{\phi}$-Nielsen class of $z$
(cf. Definition \ref{def:det}) where
$\theta_{\phi}$ is an homeomorphism isotopic to $\phi$ rel $K$
such that $(\theta_{\phi})_{|M_{1}} \in \Theta$.
\end{pro}
\begin{proof}
Proposition \ref{pro:lif2} provides a lift $\tilde{\mathcal F}$
of ${\mathcal F}$ to the universal covering $\sigma:\tilde{M} \to M$
of $M$. Suppose that $\eta(\phi)$ is not trivial, then
$\mathrm{Fix}(\tilde{\phi})$ is a non-empty compact set.
Moreover
$\sigma(\mathrm{Fix}(\tilde{\phi}))$ corresponds to the
$\theta_{\phi}$-Nielsen class rel $K$ that contains $z$.
Such a Nielsen class does not
peripherally contain any punctures (Proposition \ref{pro:lif1}).
It suffices to prove that $\mathrm{Fix}(\tilde{\mathcal F})$ is non-empty.
This is a consequence of Theorem \ref{teo:plane2}.
\end{proof}
\section{Existence of orbits of cardinality at most $2$ in the sphere}
Let us discuss briefly the proof of Theorem \ref{teo:sphere}.
It suffices to show the result for finitely generated nilpotent groups
since a family of compact sets satisfying the finite intersection property
has non-empty intersection. We will see that if Theorem \ref{teo:sphere}
holds true for groups in ${\mathcal G}_{k,p}$ then any group
$G \in {\mathcal G}_{k,p+1}$ has a finite orbit. The main ingredient of this
result is Lemma \ref{lem:fpr} below. Once the finite $G$-orbit ${\mathcal O}$
is fixed we consider the Thurston decomposition of $G$ in the marked
surface $({\mathbb S}^{2},{\mathcal O})$.
Roughly speaking we find an invariant connected component $M$ of the decomposition
and global fixed points for an irreducible nilpotent model of
a normal subgroup of index at most $2$ of $G$ (cf. Remark \ref{rem:pt2}).
Then we lift this result to the initial group by using
Proposition \ref{pro:lif2} and Theorem \ref{teo:plane}.

The following lemma is the generalization of Lemma 7.2 of \cite{FHPs}
to the nilpotent case.
\begin{lem}
\label{lem:fpr}
Let $G$ be a finitely generated nilpotent subgroup of
$\mathrm{Diff}_{+}^{1}({\mathbb S}^{2})$.
Consider a compact $G$-invariant set $X$.
Suppose that $\gamma$ is a simple closed
non-homotopically trivial curve in
${\mathbb S}^{2} \setminus X$.
Suppose that $\phi$ is isotopic rel $X$ to a homeomorphism
that preserves $\gamma$ and both components of
${\mathbb S}^{2} \setminus \gamma$ for any $\phi \in G$.
Then $G$ has a global fixed point.
\end{lem}
\begin{proof}
Let $D$ be a connected component of ${\mathbb S}^{2} \setminus \gamma$.
Consider the universal covering $Y$ of the connected component containing $D$ of
${\mathbb S}^{2} \setminus (X \setminus D)$.
Let $\tilde{D}$ be a lift of $D$ to the topological disc $Y$.
We denote by $X_{0}$ the lift of $X \cap D$ to the disc $\tilde{D}$.
There exists a unique lift $\tilde{G}$ of $G$ to $Y$ such that
$\tilde{\phi}(X_{0}) = X_{0}$ for any $\tilde{\phi} \in \tilde{G}$.
The set $X_{0}$ is compact; there exists a global fixed point for
$\tilde{G}$ (and then for $G$) by Theorem \ref{teo:plane}.
\end{proof}
Cartwright and Littlewood \cite{Cart-Little} proved that an orientation
preserving homeomorphism of ${\mathbb R}^{2}$ that leaves a
compact nonseparating continuum (a compact connected set whose complementary
is also connected)
$K$ invariant leaves a point of $K$ fixed (cf. \cite{Brown:CL} for a short proof).
We generalize Cartwright-Littlewood theorem for nilpotent subgroups of
$\mathrm{Diff}_{+}^{1}({\mathbb R}^{2})$.
\begin{proof}[Proof of Theorem \ref{teo:carlit}]
Suppose that the result is false.
There exists a simple closed curve $\gamma$ separating $K$ and $\mathrm{Fix}(G) \cup \{\infty\}$
in ${\mathbb S}^{2}$. Analogously as in Lemma \ref{lem:fpr}
we obtain a lift $\tilde{G}$ of $G$ to the universal covering of the connected component of
${\mathbb R}^{2} \setminus \mathrm{Fix}(G)$ containing $K$ such that
the lift preserves a compact set (a lift of $K$). Theorem \ref{teo:plane} implies
$\mathrm{Fix}(\tilde{G}) \neq \emptyset$ and then
$\mathrm{Fix}(G) \cap ({\mathbb R}^{2} \setminus \mathrm{Fix}(G)) \neq \emptyset$.
We obtain a contradiction.
\end{proof}
\begin{proof}[Proof of Theorem \ref{teo:sphere}]
It suffices to show the result for finitely generated subgroups.
A finitely generated nilpotent subgroup is polycyclic, so every
subgroup is finitely generated (cf. \cite[Theorem 2.7, p. 31]{Raghu}).
The theorem is proved by induction on the class of nilpotency $k$
of $G$. It is obvious for $k=0$.
Suppose that the result holds for groups whose nilpotency class is
less than or equal to $k$.
We will prove that if the theorem holds true for groups
in ${\mathcal G}_{k,p}$ then it is also satisfied for groups
in ${\mathcal G}_{k,p+1}$.
The result is obviously true for $p=0$ and $p=1$.

Let $G= \langle Z^{(k)}(G), \phi_{1}, \hdots, \phi_{p+1} \rangle$.
Our first goal is finding a finite orbit for $G$.
The groups
\[ G_{1} =  \langle Z^{(k)}(G), \phi_{1}, \hdots, \phi_{p} \rangle, \ \ \
G_{2} =  \langle Z^{(k)}(G), \phi_{p+1} \rangle \]
are normal in $G$ and have finite orbits of cardinality at most $2$
by the induction hypothesis. We define
\[ H_{1} = \langle \phi^{2} : \phi \in G_{1} \rangle , \ \ \
H_{2} =  \langle \phi^{2} : \phi \in G_{2} \rangle . \]
Notice that  the set $\{  \phi^{2} : \phi \in G_{j} \}$ is not
necessarily a subgroup of $G_{j}$.
The group $H_{j}$ is a characteristic subgroup of $G_{j}$ and then a normal subgroup of
$G$. Moreover $|G_{j}:H_{j}| < \infty$ since $G_{j}$ is polycyclic
(cf.  \cite[Lemma 4.1, p. 57]{Raghu}) for $j \in \{1,2\}$.
The sets $\mathrm{Fix}(H_{1})$ and $\mathrm{Fix}(H_{2})$
are non-empty by the existence of orbits of cardinality at most $2$ for
$G_{1}$ and $G_{2}$.
If $\mathrm{Fix}(H_{1}) \cap \mathrm{Fix}(H_{2}) \neq \emptyset$ then
$\langle H_{1},H_{2} \rangle$ has a global fixed point. Hence $G$ has a finite orbit of
cardinality at most $|G_{1}:H_{1}| |G_{2}:H_{2}|$.

Suppose $\mathrm{Fix}(H_{1}) \cap \mathrm{Fix}(H_{2}) = \emptyset$.
Let ${\mathcal R}$ be the set of reducing curves provided by Proposition
\ref{pro:tnf} ($K_{1}=\mathrm{Fix}(H_{1})$, $K_{2}=\mathrm{Fix}(H_{2})$).
If ${\mathcal R} \neq \emptyset$ there exists a finite index normal subgroup $H$
of $G$ fixing all the curves in ${\mathcal R}$ and the connected components
in ${\mathbb S}^{2} \setminus {\mathcal R}$ (modulo isotopy).
Lemma \ref{lem:fpr} implies that $H$ has a global fixed point, thus
$G$ has a finite orbit.
If ${\mathcal R} = \emptyset$ then $\sharp (\mathrm{Fix}(H_{1}) \cup \mathrm{Fix}(H_{2})) < \infty$
and $\mathrm{Fix}(H_{1})$ is a union of finite orbits of $G$.

Consider a finite $G$-orbit ${\mathcal O}$.
Let us apply Thurston normal form theorem to the group $[G]$
generated by $G$ in the mapping class group of the marked
surface $({\mathbb S}^{2},{\mathcal O})$.
Consider the tree $\Gamma$ that is dual to ${\mathcal R}$ (cf. Definition \ref{def:dual}).
The action of $G$ on $\Gamma$ has a fixed point $\tau$
(Proposition \ref{pro:iso}).

If $\tau$ is and edge we can consider a normal subgroup $H$ of $G$ of
index at most $2$ such that $H$ fixes both sides of
${\mathbb S}^{2} \setminus \tau$. The group $H$ has a global fixed point
by Lemma \ref{lem:fpr}. Hence $G$ has
either a global fixed point or a two-points orbit.

We can suppose that $G$ fixes a vertex $D$ of $\Gamma$.
Let $\hat{D}$ be the marked sphere obtained from $\overline{D}$
by contracting the curves in $\overline{D} \cap {\mathcal A}$.
Consider $[]: G \to \mathrm{MCG}(G_{|\hat{D}})$; the
model for $[G]$ provided by Lemma \ref{lem:model}
has a normal subgroup $L$ of
index at most $2$ that fixes $2$ points (Remark \ref{rem:pt2}).

We claim that the normal subgroup $H:=[]^{-1}(L)$
of index at most $2$ of $G$ has a global fixed point.
If $L$ fixes a curve in $\overline{D} \cap {\mathcal A}$
it is a consequence of Lemma \ref{lem:fpr}.
If $L$ fixes a puncture there is nothing to prove.
We can suppose that $L$ fixes an interior point.
There exists a lift
$\tilde{H}$ of $H$ to the universal covering of
${\mathbb S}^{2} \setminus {\mathcal O}$
such that the fixed point set of some element of
$\tilde{H}$ is a non-empty compact set by Proposition \ref{pro:lif2}.
Hence $\tilde{H}$ has a global fixed point by
Theorem \ref{teo:plane2}.
We obtain a global fixed point for $H$ and a finite
orbit of cardinality at most $2$ for $G$.
\end{proof}
\section{Fixed-point-free groups}
\label{section:fpf}
It is clear that the non-existence of a global fixed point for a subgroup
$G$ of $\mathrm{Diff}_{+}^{1}({\mathbb S}^{2})$ is related to torsion phenomena
(cf. Remark \ref{rem:pt2}).
So it is interesting to identify properties that hold true for  finite nilpotent
groups and to extend them for general nilpotent groups.
For instance a finite fixed-point-free nilpotent group is a dihedral
group $D_{2^{n}}$ for some $n \in {\mathbb N}$.
This group has a $2$-orbit, two $2^{n}$-orbits and every other orbit is a $2^{n+1}$-orbit.
The generalization of this property
is provided by Theorem \ref{teo:odd}: every finite orbit of a
fixed-point-free nilpotent subgroup
of $\mathrm{Diff}_{+}^{1}({\mathbb S}^{2})$ has an even number of elements.
\begin{rem}
The cardinal of every orbit of a finite
fixed-point-free nilpotent subgroup of $\mathrm{Homeo}_{+}({\mathbb S}^{2})$
is always a power of $2$.
This property does not hold for
general fixed-point-free nilpotent subgroups of $\mathrm{Diff}_{+}^{1}({\mathbb S}^{2})$.
Let $h:(0,\infty) \to {\mathbb R}$ be a $C^{\infty}$ function such that
$0 \in h(0,\infty)$, $h^{-1}(\pi)$ is a neighborhood of $0$ and $\infty$
and $h(r) + h(1/r) = 2 \pi$ for any $r \in (0,\infty)$.
We define $f(r,\theta)=(r,\theta + h(r))$ and $g(r,\theta)= (1/r,-\theta)$
in polar coordinates in ${\mathbb R}^{2}$.
The maps $f$ and $g$ extend to orientation-preserving $C^{\infty}$ diffeomorphisms of ${\mathbb S}^{2}$.
The group $G:=\langle f,g \rangle$ is abelian since $h(r) + h(1/r) \equiv 2 \pi$.
Moreover $G$ has no global fixed points but it has finite orbits of every even cardinality.
The set $\{ (1,0), (1, \pi) \}$ is a $2$-orbit.
Given $k \geq 2$ there exists $r_{0} \in (0,1)$ such that $h(r_{0}) = 2 \pi/k$.
Hence every point of the form $(r_{0},\theta)$ belongs to a orbit of cardinality
$2 k$.
\end{rem}
\begin{proof}[Proof of Theorem \ref{teo:odd}]
It suffices to prove the theorem for any finitely generated subgroup $H$ of $G$.
The set $F$ is a union of orbits of $H$, so $H$ has an orbit of odd cardinality.
We can suppose without lack of generality that $G$ is finitely generated and
that $F$ is a $G$-orbit. We also suppose that $\sharp F \geq 3$, otherwise
there is nothing to prove.

We apply the Thurston decomposition (Proposition \ref{pro:tnf})
to the marked surface $({\mathbb S}^{2},F)$.
Consider the dual tree $\Gamma$ (cf. Definition \ref{def:dual}).
Either there exists a connected component $D$ of ${\mathbb S}^{2} \setminus {\mathcal A}$
that is $G$-invariant modulo isotopy rel $F$ or there exists an invariant element
$\gamma$ of ${\mathcal R}$ such that the components of ${\mathbb S}^{2} \setminus {\mathcal R}$
are permuted by $G$. The latter situation is impossible since
$\sharp F$ is odd.

Consider the surface $\hat{D}$ obtained by contracting the curves
in $\overline{D} \cap {\mathcal A}$.
If $D \cap F \neq \emptyset$ then $F$ is a subset of $D$ and there are no reducing
curves. In particular we have $\hat{D} = ({\mathbb S}^{2},F)$. We denote $\hat{F}=F$.
If $D \cap F = \emptyset$ we define $\hat{F}$ as the set of points in $\hat{D}$
corresponding to curves in $\overline{D} \cap {\mathcal A}$.
Consider a curve $\gamma$ in $\overline{D} \cap {\mathcal A}$.
The connected component $D_{\gamma}$ of ${\mathbb S}^{2} \setminus \gamma$
such that $D_{\gamma} \cap D = \emptyset$ contains $p \geq 2$ points.
The action of $G$ on $F$ is transitive, hence $p$ does not depend on $\gamma$
and $\sharp F = p q$ where $q \geq 3$ is the cardinal of $\hat{F}$.
We deduce that $\sharp \hat{F}$ is odd.

Let $[]: G \to \mathrm{MCG}(\hat{D})$ be the canonical projection.
Consider a nilpotent model $G_{0}$ of $[G]$ whose elements are
all irreducible (Lemma \ref{lem:model}).
By construction $\hat{F}$ is an orbit of odd cardinality
of $G_{0}$. Let us show that $G_{0}$
has a global fixed point.
The group $G_{0}$ is of the form $\langle \mathrm{Tor}(G_{0}), f_{0} \rangle$ where
$\mathrm{Tor}(G_{0})$ is the normal subgroup of finite order elements of
$G_{0}$ and either $f_{0} \equiv Id$ or $f_{0}$ is pseudo-Anosov (Lemma \ref{lem:model}).
The existence of a global fixed point is obvious if
$\mathrm{Tor}(G_{0}) = \{Id\}$.

Suppose $\mathrm{Tor}(G_{0}) \neq \{Id\}$. Thus there exists
$h \in \mathrm{Tor}(G_{0}) \cap Z^{(1)} (G_{0})$ with $h \neq Id$.
If the two points in $\mathrm{Fix} (h)$ are not fixed points of $G_{0}$
then there exists $h' \in G_{0}$ permuting the points in $\mathrm {Fix}(h)$
(and commuting with $h$). A theorem of Handel (cf. Lemma \ref{lem:Handel})
implies $h^{2} \equiv Id$. The set $\hat{F}$ is a union of  orbits of $h$.
Since $\sharp \hat{F}$ is odd, it contains exactly an element of $\mathrm {Fix}(h)$.
Clearly it is a global fixed point.

Let us remark that there are no fixed points of
$G_{0}$ in $\hat{F}$.
Since $[G] \neq \{Id\}$ (the action on punctures and reducing curves is not trivial),
there exists a lift $\tilde{G}$ of $G$ to the universal covering of
${\mathbb S}^{2} \setminus F$ with
an element $\tilde{\alpha} \in \tilde{G}$ satisfying that
$\mathrm{Fix}(\tilde{\alpha})$ is non-empty and compact (Proposition \ref{pro:lif2}).
Hence Theorem \ref{teo:plane2} implies that
$\mathrm{Fix}(\tilde{G})$ (and then $\mathrm{Fix}(G)$) is not empty.
\end{proof}
\subsection{Conditions on $2$-orbits}
\label{subsec:cond}
We study the structure of the $2$-orbits of a nilpotent subgroup
$G$ of $\mathrm{Diff}_{+}^{1}({\mathbb S}^{2})$ and no global fixed
points.

Let us begin the section with some examples of finite groups.
Consider the subgroup $\langle 1/z, \lambda z \rangle$ of
$\mathrm{PGL}(2,{\mathbb C})$ where $\lambda$ is a primitive
$2^{n}$th root of unity. This group is isomorphic to the dihedral group
$D_{2^{n}}$.
Every finite fixed-point-free nilpotent group of orientation-preserving homeomorphisms
is topologically conjugated to $D_{2^{n}}$ for some
$n \in {\mathbb N}$ (cf. Remark \ref{rem:pt2}).
The group $D_{2}$ is commutative and has three $2$-orbits
$\{0, \infty\}$, $\{1,-1\}$, $\{i,-i\}$.
The group $D_{2^{n}}$ is not commutative,
has nilpotency class equal to $n$ and
a unique $2$-orbit
$\{0, \infty\}$ for any $n \geq 2$. We see in this section that nilpotent
subgroups of $\mathrm{Diff}_{+}^{1}({\mathbb S}^{2})$ with no global fixed
points share analogous properties.
\begin{defi}
Consider $2$-orbits ${\mathcal O}_{1}$ and ${\mathcal O}_{2}$ of a
subgroup $G$ of $\mathrm{Homeo}_{+}({\mathbb S}^{2})$.
We say that $[{\mathcal O}_{1}]=[{\mathcal O}_{2}]$ if we have
\[ \phi_{|{\mathcal O}_{1}} \equiv Id \Leftrightarrow \phi_{|{\mathcal O}_{2}} \equiv Id \]
for any $\phi \in G$.
\end{defi}
\begin{defi}
\label{def:tau}
Consider $2$-orbits ${\mathcal O}_{1}$, $\hdots$, ${\mathcal O}_{n}$  of a
subgroup $G$ of $\mathrm{Homeo}_{+}({\mathbb S}^{2})$.
The mapping
\[
\begin{array}{ccccc}
\tau & : & G & \to & ({\mathbb Z}/2 {\mathbb Z})^{n} \\
& & \phi & \mapsto & (a_{1},\hdots,a_{n}) \\
\end{array}
\]
where $a_{j}=0$ if and only if $\phi_{|{\mathcal O}_{j}} \equiv Id$ is a morphism of groups.
We say that ${\mathcal O}_{1}$, $\hdots$, ${\mathcal O}_{n}$ are {\it independent} if the action of $G$ on
${\mathcal O}_{1} \cup \hdots \cup {\mathcal O}_{n}$ is
$({\mathbb Z}/2 {\mathbb Z})^{n}$, i.e. if $\tau(G)$ is equal to $({\mathbb Z}/2 {\mathbb Z})^{n}$.
\end{defi}
\begin{rem}
If the classes of $2$-orbits ${\mathcal O}_{1}$, ${\mathcal O}_{2}$
are different ($[{\mathcal O}_{1}] \neq [{\mathcal O}_{2}]$) then the action of
$G$ on ${\mathcal O}_{1} \cup {\mathcal O}_{2}$ is
${\mathbb Z}/2 {\mathbb Z} \times {\mathbb Z}/2 {\mathbb Z} = D_{2}$.
\end{rem}
\begin{rem}
Two different classes of $2$-orbits are always independent.
This is not the case for $3$ classes of $2$-orbits.
For instance the subgroup $D_{2}= \langle 1/z,-z \rangle$ of
$\mathrm{PGL}(2,{\mathbb C})$ has three $2$-orbits
whose classes are pairwise different
but they are not independent since the action of $G$
on the union of the orbits is $D_{2}$.
\end{rem}
\begin{lem}
\label{lem:4t3i}
Consider $2$-orbits ${\mathcal O}_{1}$, ${\mathcal O}_{2}$, ${\mathcal O}_{3}$, ${\mathcal O}_{4}$
of a subgroup $G$ of $\mathrm{Homeo}_{+}({\mathbb S}^{2})$ whose classes are pairwise
different. Then there are three of them that are independent.
\end{lem}
\begin{proof}
Suppose that there is no subset of $3$ independent $2$-orbits.
We can consider the action of $G$ on
${\mathcal O}_{1} \cup {\mathcal O}_{2} \cup {\mathcal O}_{3} \cup {\mathcal O}_{4}$
as a subgroup of $({\mathbb Z}/2 {\mathbb Z})^{4}$. Indeed consider the mapping
$\tau$ in Definition \ref{def:tau} for $n=4$.
If an element $a$ of $\tau (G)$ has $2$ coordinates $a_{j}$ and $a_{k}$ ($j \neq k$)
equal to $0$ then any other coordinate $a_{l}$ of $a$ is equal to $0$. Otherwise
${\mathcal O}_{j}$, ${\mathcal O}_{k}$ and ${\mathcal O}_{l}$ are independent.
The projection of $\tau (G)$ in any $2$ coordinates is equal to
$({\mathbb Z}/2 {\mathbb Z})^{2}$ since different $2$-orbits are independent.
We deduce that $(0,1,1,1)$ and $(1,0,1,1)$ belong to $\tau (G)$.
Hence its sum  $(1,1,0,0)$ belongs to $\tau (G)$ providing an element with exactly
two vanishing coordinates. We obtain a contradiction.
\end{proof}
We already know that $D_{2}= \langle 1/z,-z \rangle$ and $D_{4}= \langle 1/z, i z \rangle$
have $3$ classes and $1$ class of $2$-orbits
respectively. Moreover
the unique commutative finite group in
$\mathrm{Homeo}_{+}({\mathbb S}^{2})$ with no global fixed point is
$D_{2}$ that has $3$ classes of
$2$-orbits. These properties are generalized in Theorem \ref{teo:scar}.
The next result is a corollary of Theorem \ref{teo:scar}.
\begin{cor}
Let $G \subset \mathrm{Diff}_{+}^{1}({\mathbb S}^{2})$ be a nilpotent group
whose number of classes of $2$-orbits is $0$, $2$ or greater than
or equal to $4$. Then $G$ has a global
fixed point.
\end{cor}
Given $\phi, \eta$ commuting elements of $\mathrm{Homeo}_{+}({\mathbb S}^{2})$
there exist isotopies $\phi_{t}$ and $\eta_{t}$ such that
$\phi_{0}=\eta_{0}=Id$, $\phi_{1}=\phi$, $\eta_{1}=\eta$.
The invariant $w(\phi,\eta)$ is the homotopy class defined by
$\gamma_{t}=\phi_{t}^{-1} \eta_{t}^{-1} \phi_{t} \eta_{t}$ in
$\pi_{1}(\mathrm{Homeo}_{+}({\mathbb S}^{2}), Id)$.
This group is isomorphic to ${\mathbb Z}/2 {\mathbb Z}$
since $\mathrm{Homeo}_{+}({\mathbb S}^{2})$ is
homotopically equivalent to $\mathrm{SO}(3)$ \cite{Kneser}.
Then we can analyze the existence
and properties of $2$-orbits of commutative groups by using Theorem \ref{teo:FHP}.
Even if a nilpotent group
$G \subset \mathrm{Diff}_{+}^{1}({\mathbb S}^{2})$
(with no global fixed point) is not in general
commutative we can still apply the properties of the invariant $w$
(cf. \cite{Handel-top} \cite{FHPs}) since its image in
a certain mapping class group is abelian.
\begin{lem}
\label{lem:ircon}
Let $G \subset \mathrm{Diff}_{+}^{1}({\mathbb S}^{2})$ be a
fixed-point-free nilpotent group.
Let $F$ be a finite set composed by two or three
$2$-orbits of $G$ whose classes are pairwise different. Then the image $[G]$ of $G$ by
the mapping $[]: G \to \mathrm{MCG}({\mathbb S}^{2},F)$ is a commutative irreducible group.
\end{lem}
\begin{proof}
We apply the Thurston decomposition (Proposition \ref{pro:tnf}) to
the marked surface $({\mathbb S}^{2},F)$ with
$H_{1}=H_{2}=G$, $K_{1}=\emptyset$, $K_{2}=\emptyset$.
Suppose that $[G]$ is not irreducible; thus
the set ${\mathcal R}$ of reducing curves is non-empty.
Consider the tree $\Gamma$ that is dual to ${\mathcal R}$.
There is no invariant element
$\gamma$ of ${\mathcal R}$ such that the components of ${\mathbb S}^{2} \setminus {\mathcal R}$
are permuted by $G$ (modulo isotopy rel $F$)
since such property implies that all the $2$-orbits in
$F$ are equivalent.
Thus there exists a connected component $D$ of ${\mathbb S}^{2} \setminus {\mathcal A}$
that is $G$-invariant modulo isotopy.

We claim that $G$ fixes a curve in $\overline{D} \cap {\mathcal A}$ modulo isotopy rel $F$.
Let $D'$ be a connected component of ${\mathbb S}^{2} \setminus \overline{D}$.
The set $D'$ contains at least $2$ points of $F$
since curves in ${\mathcal R}$ are essential.
If $F \cap D'$ does not contain a $2$-orbit then there are two equivalent
$2$-orbits in $F$. This contradicts the choice of $F$.
Hence $F \cap D'$ contains a $2$-orbit and
$\partial{D} \cap \partial{D'}$ is $G$-invariant modulo isotopy rel $F$.
Lemma \ref{lem:fpr} implies that $G$ has a global fixed point, contradicting the hypothesis.
We deduce ${\mathcal R}=\emptyset$.

Consider a nilpotent model $G_{0}$ of $[G]$ whose elements are
all irreducible (Lemma \ref{lem:model}).
The set $\mathrm{Tor}(G_{0})$ of finite order elements is a subgroup of
$G_{0}$ (Lemma \ref{lem:model}).
It is either a cyclic group $C_{n}$ or a dihedral group $D_{2^{n}}$ for some
$n \in {\mathbb N}$ (cf. Remark \ref{rem:pt2}).
The set $F$ contains at least two disjoint invariant sets of $2$ elements of
$\mathrm{Tor}(G_{0})$. Hence we obtain
$\mathrm{Tor}(G_{0})=\{Id\}$, $\mathrm{Tor}(G_{0})=C_{2}$ or $\mathrm{Tor}(G_{0})=D_{2}$.
In the first case $G_{0}$ is clearly commutative.
In the second case
$Z^{(1)}(G_{0}) \cap \mathrm{Tor}(G_{0}) \neq \{Id\}$ implies
$\mathrm{Tor}(G_{0}) \subset Z^{(1)}(G_{0})$.
Since $G_{0}/\mathrm{Tor}(G_{0})$ is cyclic, $G_{0}$ is commutative.
We can suppose  $\mathrm{Tor}(G_{0})=D_{2}$.
The group $\mathrm{Tor}(G_{0})$ has exactly three $2$-orbits. Each one of them
is the fixed point set of one of the three non-trivial elements of $\mathrm{Tor}(G_{0})$.
Let $h \in \mathrm{Tor}(G_{0}) \setminus \{Id\}$.
If $f h f^{-1} = h' \neq h$ for some $f \in G_{0}$ then
$\mathrm{Fix}(h) \cup \mathrm{Fix}(h')$ is contained in an orbit of $G_{0}$
and $G_{0}$ has at most an invariant set of $2$ elements.
We deduce $\mathrm{Tor}(G_{0}) \subset Z^{(1)}(G_{0})$; the group $G_{0}$ is commutative.
\end{proof}
\begin{defi}
Consider the setup in Lemma \ref{lem:ircon}.
We can define an invariant $w_{F}:G \times G \to {\mathbb Z}/2{\mathbb Z}$.
Indeed given $\phi, \eta \in G$ we consider
elements $\phi_{0}, \eta_{0} \in \mathrm{Homeo}_{+}({\mathbb S}^{2})$ such that
$[\phi]=[\phi_{0}]$ and $[\eta]=[\eta_{0}]$ in $\mathrm{MCG}({\mathbb S}^{2},F)$
and $\phi_{0} \eta_{0}=\eta_{0} \phi_{0}$.
We define $w_{F}(\phi,\eta)=w(\phi_{0},\eta_{0})$.
\end{defi}
\begin{rem}
Since the center of the group $\pi_{1}({\mathbb S}^{2} \setminus F)$ is trivial
(${\mathbb S}^{2} \setminus F$ has negative Euler characteristic), the group
$\pi_{1}(\mathrm{MCG}({\mathbb S}^{2},F),Id)$ is also trivial.
In particular the definition of $w_{F}(\phi,\eta)$ does not depend on the
choices of $\phi_{0}$ and $\eta_{0}$ (provided they commute).
We can choose $\phi_{0}$ and $\eta_{0}$ in the group $G_{0}$ defined in Lemma
\ref{lem:ircon}.
\end{rem}
The next lemmas are used to obtain properties of the invariant $w_{F}$.
They provide the algebraic structure of the invariant and
link the existence of fixed points with its vanishing
respectively.
\begin{pro}
\label{pro:32}
\cite{FHPs}[Proposition 3.2]
Let $G$ be a commutative subgroup of
$\mathrm{Homeo}_{+}({\mathbb S}^{2})$.
The invariant $w: G \times G \to {\mathbb Z}/2{\mathbb Z}$ is symmetric
and $w(.,g)$ is a morphism of groups for any $g \in G$.
\end{pro}
\begin{lem}
\label{lem:Handel}
\cite{Handel-top}[Lemma 1.2]
Suppose that $\eta: {\mathbb S}^{2} \to {\mathbb S}^{2}$ is a rotation about some axis by $2 \pi p/q$
($p/q \in {\mathbb Q}, 0 < p/q < 1$) and that $\phi \in \mathrm{Homeo}_{+} ({\mathbb S}^{2})$
commutes with $\eta$.
\begin{itemize}
\item If $\phi$ interchanges the fixed
points $z_{1}$ and $z_{2}$ of $\eta$ then $p/q = 1/2$ and $w(\phi, \eta) = 1$.
\item If $\phi(z_{1})=z_{1}$ and $\phi(z_{2})=z_{2}$ then $w(\phi, \eta) = 0$.
\end{itemize}
\end{lem}
\begin{proof}
The first case is Handel's lemma. The second property is an easy exercise whose proof
we include for the sake of clarity.

Let $A$ be the open annulus ${\mathbb S}^{2} \setminus \{z_{1},z_{2}\}$.
Consider the universal covering $\tilde{A}$ of $A$ and a generator
$T$ of the group of covering transformations.
There exists an isotopy $\eta_{t}$ rel $\{z_{1},z_{2}\}$ such that
$\eta_{0}=Id$ and $\eta_{1}=\eta$. For instance $\eta_{t}$ can be a rotation by
$2 \pi t p/q$.
Consider the lift $\tilde{\eta}_{t}$ of $\eta_{t}$ to $\tilde{A}$ such that $\tilde{\eta}_{0}=Id$.
Since the class of an element in the mapping class group of the twice-punctured
sphere is determined by the action on the punctures
($\mathrm{MCG}(S_{0,2}) \approx {\mathbb Z}/2 {\mathbb Z}$),
we consider an isotopy $\phi_{t}$ rel $\{z_{1},z_{2}\}$
such that $\phi_{0}=Id$ and $\phi_{1}=\phi$. We consider the lift
$\tilde{\phi}_{t}$ of $\phi_{t}$ to $\tilde{A}$ such that $\tilde{\phi}_{0}=Id$.
We have $\tilde{\eta}^{q} = T^{p}$, $\tilde{\phi} T = T \tilde{\phi}$, $\tilde{\eta} T = T \tilde{\eta}$
and $\tilde{\eta} \tilde{\phi} = T^{a} \tilde{\phi} \tilde{\eta}$ for some $a \in {\mathbb Z}$.

We have
\[ T^{p} \tilde{\phi} = \tilde{\eta}^{q} \tilde{\phi} = T^{aq} \tilde{\phi} \tilde{\eta}^{q} =
T^{aq} \tilde{\phi} T^{p} = T^{aq+p} \tilde{\phi} \implies p = aq+p \implies a=0. \]
The class of the path
$\gamma_{z}:[0,1] \to A$ defined by
$\gamma(t) = (\phi_{t}^{-1} \eta_{t}^{-1} \phi_{t} \eta_{t})(z)$
is trivial in $\pi_{1}(A,z)$ for any $z \in A$ since
$\tilde{\eta} \tilde{\phi} = \tilde{\phi} \tilde{\eta}$.
The path $\phi_{t}^{-1} \eta_{t}^{-1} \phi_{t} \eta_{t}$ is trivial in
$\pi_{1}(\mathrm{Homeo}_{+}({\mathbb S}^{2},\{z_{1},z_{2}\}),Id)$
(cf. \cite{Handel-top}[Lemma 1.1]). We deduce $w(\phi,\eta)=0$.
\end{proof}
Next proposition encapsulates the properties of the invariant $w_{F}$
that allow to prove Theorem \ref{teo:scar}.
\begin{pro}
\label{pro:setup}
Let $G \subset \mathrm{Diff}_{+}^{1}({\mathbb S}^{2})$ be a
fixed-point-free nilpotent group.
Let $F$ be a finite set composed by two or three
$2$-orbits of $G$ whose classes are pairwise different. Then we have
\begin{itemize}
\item The invariant $w_{F}: G \times G \to {\mathbb Z}/2{\mathbb Z}$ is symmetric
and $w_{F}(.,\eta)$ is a morphism of groups for any $\eta \in G$.
\item Given a subgroup $H$ of $G$ such that  $(w_{F})_{|H \times H} \equiv 0$
then $H$ has a global fixed point. In particular we have $w_{F} \not \equiv 0$.
\end{itemize}
\end{pro}
\begin{proof}
The image $[G]$ of $G$ in $\mathrm{MCG}({\mathbb S}^{2},F)$ is an irreducible
commutative group by Lemma \ref{lem:ircon}. Let $G_{0}$ be the model of
$[G]$ provided by Lemma \ref{lem:model}. Since
$w:G_{0} \times G_{0} \to {\mathbb Z}/2{\mathbb Z}$ is symmetric and
a morphism of groups in each component (Proposition \ref{pro:32}), so is $w_{F}$.

Consider the subgroup $H_{0}$ of $G_{0}$ that is the image of $[H]$ by the
isomorphism $[G] \to G_{0}$ that associates elements defining the same class in
$\mathrm{MCG}({\mathbb S}^{2},F)$.

We claim that $\mathrm{Fix}(H_{0})$ is not empty.
If $\mathrm{Tor}(H_{0}) = \{Id\}$ then $H_{0}$ is cyclic and $\mathrm{Fix}(H_{0}) \neq \emptyset$.
Let $h \in \mathrm{Tor}(H_{0}) \setminus \{Id\}$.
The set $\mathrm{Fix}(h)$ has $2$ points and it is $H_{0}$-invariant.
Lemma \ref{lem:Handel} implies $\mathrm{Fix}(h) = \mathrm{Fix}(H_{0})$.

We can suppose $\mathrm{Fix}(H_{0}) \cap F = \emptyset$ since
$\mathrm{Fix}(H_{0}) \cap F = \mathrm{Fix}(H) \cap F$.
There exists a lift $\tilde{H}$ of $H$ to the universal covering of
${\mathbb S}^{2} \setminus F$ with an element
$\tilde{\alpha} \in \tilde{H}$ such that
$\mathrm{Fix}(\tilde{\alpha})$ is a non-empty compact set
(Proposition \ref{pro:lif2}).
Theorem \ref{teo:plane2} implies
that $\mathrm{Fix}(\tilde{H})$ and then $\mathrm{Fix}(H)$
are not empty.
\end{proof}
Consider the hypotheses in Proposition \ref{pro:setup}.
In next lemma we see that $w_{F}$ is completely determined.
The description of $w_{F}$ implies Theorem \ref{teo:scar}
in a straightforward way.
Moreover the mapping $w_{F}$ is simple, for instance
$w_{F}$ factors through a quotient
(isomorphic to $D_{2}$) of $G$.
\begin{lem}
\label{lem:J}
Let $G \subset \mathrm{Diff}_{+}^{1}({\mathbb S}^{2})$ be a
fixed-point-free nilpotent group.
Let $F$ be a finite set composed by two or three
$2$-orbits of $G$ whose classes are pairwise different.
Fix different $2$-orbits ${\mathcal O}_{1}$, ${\mathcal O}_{2}$ contained in $F$.
We denote
\[ H_{j} = \{ h \in G : h_{|{\mathcal O}_{j}} \equiv Id \}  \ \mathrm{for} \ j \in \{1,2\}. \]
Then
$H_{1} \cap H_{2}$ coincides with
$J:= \{ \phi \in G : w_{F}(\phi,\eta)=0 \ \forall \eta \in G \}$.
Moreover we have
\begin{itemize}
\item If $F$ is composed of two $2$-orbits then there exists a third class of $2$-orbits for $G$.
\item If $F$ is composed of three $2$-orbits then they are not independent.
\end{itemize}
\end{lem}
\begin{proof}
Consider the image $[G]$ of $G$ in $\mathrm{MCG}({\mathbb S}^{2},F)$.
There exists a realization $\sigma: [G] \to G_{0}$ of $[G]$
as an irreducible group (Lemma \ref{lem:model}).

The group $H_{j}$ is a normal subgroup of index $2$ of $G$ for $j \in \{1,2\}$.
The groups $H_{1}$ and $H_{2}$ are different since $[{\mathcal O}_{1}] \neq [{\mathcal O}_{2}]$.
Consider $\eta_{1} \in H_{2} \setminus H_{1}$ and $\eta_{2} \in H_{1} \setminus H_{2}$.
The group $\sigma (H_{j})$ has global fixed points in $F$ and
$\sigma (H_{j})/\mathrm{Tor}(\sigma (H_{j}))$ is cyclic.
Lemma \ref{lem:Handel} implies that
$w_{|\sigma (H_{j}) \times \sigma(H_{j})}$ is trivial and so is
$(w_{F})_{|H_{j} \times H_{j}}$.

Consider $\phi \in H_{1} \cap H_{2}$, we have
that $w_{F}(\phi,\alpha_{1})=0$ for any $\alpha_{1} \in H_{1}$ since $\phi \in H_{1}$ and
$w_{F}(\phi, \eta_{1})=0$ since $\eta_{1}, \phi \in H_{2}$.
Since $w_{F}(\phi,.)$ is a morphism of groups and $G= \langle H_{1},\eta_{1} \rangle$, we obtain that
$\phi$ belongs to $J$. Therefore $H_{1} \cap H_{2}$ is contained in $J$.
The mapping $(w_{F})_{|G \times G}$
is not trivial but $(w_{F})_{|H_{1} \times H_{1}}$ is.
Thus there exists $\beta_{1} \in H_{1}$ such that $w_{F}(\beta_{1},\eta_{1})=1$.
Consider the normal subgroup
\[ J_{1} = \{ \eta \in G : w_{F}(\eta,\beta_{1})=0 \}  \]
of $G$.
It is clear that $H_{1} \subset J_{1}$, $J_{1} \neq G$ and that
$J_{1}$ has index $2$. Thus $H_{1}$ and $J_{1}$ coincide and
$J$ is contained in $H_{1}$. Analogously we obtain
$J \subset H_{2}$. We get $J \subset H_{1} \cap H_{2}$ and then $J = H_{1} \cap H_{2}$.

Notice that if ${\mathcal O}_{3}$ is a third $2$-orbit in $F$ the previous proof shows that
$J = H_{1} \cap H_{2} \cap H_{3}$ and the orbits of $F$ can not be independent.

Suppose that $F = {\mathcal O}_{1} \cup {\mathcal O}_{2}$.
The group $G/J = \{ Id, [\eta_{1}], [\eta_{2}], [\eta_{1} \eta_{2}]\}$ is isomorphic to $D_{2}$.
We define $H_{3} = \langle J, \eta_{1} \eta_{2} \rangle$, it is a normal subgroup
of index $2$ of $G$ such that
$w_{F}:H_{3} \times H_{3} \to {\mathbb Z}/2{\mathbb Z}$ is the trivial mapping.
Then $H_{3}$ has a global fixed point by Proposition \ref{pro:setup}.
The group
$H_{3}$ has no fixed points in $F$ since $\eta_{1} \eta_{2}$
does not fix any element of $F$.
Since $H_{3}$ is normal of index $2$ in $G$, there exists
a $2$-orbit ${\mathcal O}_{3} \subset \mathrm{Fix}(H_{3})$.
We have $(\eta_{1}\eta_{2})_{|{\mathcal O}_{3}} \equiv Id$,
$(\eta_{1}\eta_{2})_{|{\mathcal O}_{1}} \neq Id$ and
$(\eta_{1}\eta_{2})_{|{\mathcal O}_{2}} \neq Id$.
Thus we obtain
$[{\mathcal O}_{1}] \neq [{\mathcal O}_{3}]$ and
$[{\mathcal O}_{2}] \neq [{\mathcal O}_{3}]$.
\end{proof}
\begin{cor}
\label{cor:noind3}
Let $G \subset \mathrm{Diff}_{+}^{1}({\mathbb S}^{2})$ be a
fixed-point-free nilpotent group.
Then there are no $3$ independent $2$-orbits.
\end{cor}
\begin{proof}[Proof of Theorem \ref{teo:scar}]
The number of classes of $2$-orbits is never equal to $2$ since
the existence of two $2$-orbits implies that there is a third one
by Lemma \ref{lem:J}.
Moreover such a number can not be greater than $3$ by Lemma
\ref{lem:4t3i} and Corollary \ref{cor:noind3}.

Suppose that $G$ is commutative.
Let ${\mathcal O}_{1}$ be a $2$-orbit of $G$.
We denote $H_{1} =  \{ \phi \in G : \phi_{|{\mathcal O}_{1}} \equiv Id \}$.
Since ${\mathcal O}_{1}$ is contained in $\mathrm{Fix}(H_{1})$, we deduce that
$w: H_{1} \times H_{1} \to {\mathbb Z}/2{\mathbb Z}$ is the trivial mapping
(Theorem \ref{teo:FHP}).
Theorem \ref{teo:FHP} also implies that $w:G \times G \to {\mathbb Z}/2{\mathbb Z}$
is not trivial. Let $\eta_{1} \in G \setminus H_{1}$.
Since $H_{1}$ is an index $2$ normal subgroup of $G$ and $w(\eta_{1},\eta_{1})=0$,
there exists $\alpha_{1} \in H_{1}$ such that $w(\alpha_{1},\eta_{1})=1$. We define
\[ H_{2} = \{ \phi \in G : w (\phi, \eta_{1})=0 \} . \]
The group $H_{2}$ has index $2$ since
$\alpha_{1} \in H_{1} \setminus H_{2}$.
Thus $J:= H_{1} \cap H_{2}$ is a (normal) subgroup of $G$ of index $4$.
Indeed we have $H_{2} = \langle J, \eta_{1} \rangle$.
Since $w_{|H_{1} \times H_{1}} \equiv 0$ and $w_{|H_{2} \times \{\eta_{1}\}} \equiv 0$,
we deduce $w_{|H_{2} \times H_{2}} \equiv 0$.
Hence the set $\mathrm{Fix}(H_{2})$ is non-empty (Theorem \ref{teo:FHP}).
There exists a $2$-orbit ${\mathcal O}_{2}$ of $G$ contained in $\mathrm{Fix}(H_{2})$.
Since $(\eta_{1})_{|{\mathcal O}_{2}} \equiv Id$ and $(\eta_{1})_{|{\mathcal O}_{1}} \not \equiv Id$,
we obtain $[{\mathcal O}_{1}] \neq [{\mathcal O}_{2}]$.
The group $G$ has $3$ classes of $2$-orbits by the first part of the proof.
\end{proof}
Theorems \ref{teo:odd} and \ref{teo:scar}
for commutative groups admit simple proofs by using only the characterization
of commutative subgroups $G$ of $\mathrm{Diff}_{+}^{1}({\mathbb S}^{2})$ with
global fixed points of \cite{FHPs}
(cf. Theorem \ref{teo:FHP}).
\begin{cor}
Let $G \subset \mathrm{Diff}_{+}^{1}({\mathbb S}^{2})$ be a
fixed-point-free nilpotent group.
Let $J$ be the subgroup of elements of $G$
fixing every  $2$-orbit pointwise.
Then it is a normal subgroup
such that $G/J \approx {\mathbb Z}/2{\mathbb Z}$ or
$G/J \approx D_{2}$ depending on whether there are
$1$ or $3$ classes of $2$-orbits.
\end{cor}
\bibliography{rendu}
\end{document}